\documentclass[reqno,12pt]{amsart}

\usepackage[active]{srcltx}
\usepackage{color}
\usepackage{amsmath,amsthm,amssymb}
\usepackage{epsfig}
\usepackage{graphicx}
\usepackage{amssymb}
\usepackage{latexsym}
\usepackage{pdfpages}
\usepackage{mathtools}
\usepackage{enumitem}

\usepackage{ulem}

\usepackage{ifpdf}

\mathtoolsset{showonlyrefs}

\def\z{{\bf z}}
\def\a{{\bf a}}

\def\divi{\hbox{\rm div\,}}

\ifpdf %if using pdfLaTeX in PDF mode
\usepackage[hidelinks]{hyperref}
\else %if using LaTeX or pdfLaTeX in DVI mode
\fi

\usepackage[usenames,dvipsnames]{pstricks}
\usepackage{epsfig}
\usepackage{pst-grad} % For gradients
\usepackage{pst-plot} % For axes

\usepackage{color}

\newtheorem{theorem}{Theorem}[section]
\newtheorem{lemma}[theorem]{Lemma}
\newtheorem{definition}[theorem]{Definition}
\newtheorem{proposition}[theorem]{Proposition}

\newtheorem{remark}[theorem]{Remark}
\newtheorem{example}[theorem]{Example}
\newtheorem*{theorem*}{\it Theorem}

\def\divi{\hbox{\rm div\,}}

\def\R{\mathbb R}

\numberwithin{equation}{section}

\def\1{\raisebox{2pt}{\rm{$\chi$}}}

\def\XXint#1#2#3{{\setbox0=\hbox{$#1{#2#3}{\int}$}
		\vcenter{\hbox{$#2#3$}}\kern-.5\wd0}}

\makeatletter
\newcommand{\labeltext}[2]{%
  \@bsphack
  \csname phantomsection\endcsname % in case hyperref is used
  \def\@currentlabel{#1}{\label{#2}}%
  \@esphack
}
\makeatother

\begin{document}
	
\title[Duality approach to gradient flows]{\bf A duality-based approach to gradient flows of linear growth functionals}
	
\author[W. G\'{o}rny and J. M. Maz\'on]{Wojciech G\'{o}rny and Jos\'e M. Maz\'on}
	
\address{ W. G\'{o}rny: Faculty of Mathematics, Universit\"at Wien, Oskar-Morgerstern-Platz 1, 1090 Vienna, Austria; Faculty of Mathematics, Informatics and Mechanics, University of Warsaw, Banacha 2, 02-097 Warsaw, Poland
\hfill\break\indent
{\tt  wojciech.gorny@univie.ac.at }
}
	
\address{J. M. Maz\'{o}n: Departamento de An\`{a}lisis Matem\`atico,
Universitat de Val\`encia, Dr. Moliner 50, 46100 Burjassot, Spain.
\hfill\break\indent
{\tt mazon@uv.es }}
	
%
%%\hfill\break\indent
%
	
	
\setcounter{tocdepth}{1}

\date{\today}
	
\begin{abstract}
We study gradient flows of general functionals with linear growth with very weak assumptions. Classical results concerning characterisation of solutions require differentiability of the Lagrangian, as for the time-dependent minimal surface equation, or a special form of the Lagrangian as in the total variation flow. We propose to study this problem using duality techniques, give a general definition of solutions and prove their existence and uniqueness. This approach also allows us to reduce the regularity and structure assumptions on the Lagrangian.
\end{abstract}
	
\maketitle
%
%{\red Maz\'on is  working.}
%
%{\blue Wojciech is NOT working.}

{\renewcommand\contentsname{Contents}
\setcounter{tocdepth}{3}
\tableofcontents}

\section{Introduction}\label{Introd}

In this paper, we are interested in the study of gradient flows of general functionals with linear growth. The setting is as follows: suppose that $\Omega \subset \mathbb{R}^N$ is an open bounded set with Lipschitz boundary and that $f \in C(\overline{\Omega} \times \mathbb{R}^N)$ is a convex function. Then, we study the gradient flow of the functional
\begin{equation}
F(u) = \int_\Omega f(x,Du)
\end{equation}
defined over $BV(\Omega) \cap L^2(\Omega)$, i.e. the evolution equation
\begin{equation}
\left\{ \begin{array}{lll} u_t (t,x) = {\rm div} (\partial_\xi f(x, Du(t,x)))   \quad &\hbox{in} \ \ (0, T) \times \Omega; \\[5pt] u(0,x) = u_0(x) \quad & \hbox{in} \ \  \Omega, \end{array} \right.
\end{equation}
where $u_0 \in L^2(\Omega)$, subject to Neumann or Dirichlet boundary conditions. Here, $\partial_\xi f$ denotes the subdifferential of $f$ in the second variable. A typical example of a Lagrangian $f(x, \xi)$ in this framework  is the nonparametric area integrand $f(x, \xi) = \sqrt{1 + \Vert \xi \Vert^2}$, whose associated  problem is the time-dependent minimal surface equation, which has been studied in \cite{DemengelTeman1} and \cite{LT}. The corresponding steady-state problem is the {\it nonparametric Plateau problem}
\begin{equation}\label{Plateau}
\left\{ \begin{array}{ll} - {\rm div} \left( \frac{\nabla u}{\sqrt{1 + \vert \nabla u \vert^2}} \right) = 0 \quad &\hbox{in} \quad \Omega; \\[10pt] u = \varphi \quad &\hbox{on} \quad \partial \Omega. \end{array}  \right.
\end{equation}
This problem was studied, using a combination of techniques from geometric measure theory and PDEs, by many authors including Federer, Fleming, De Giorgi, Bombieri, Giusti, Miranda, Nitsche and Serrin. A good reference are the monographs \cite{DHKW} and \cite{Giusti}. Another example is the evolution problem for plastic antiplanar shear studied in \cite{Zhou}, which corresponds to the plasticity functional $f$ given by
$$ f(\xi) = \left\{ \begin{array}{ll} \frac{1}{2} \Vert \xi \Vert^2 & \
\ \ {\rm
if} \ \ \ \Vert \xi \Vert \leq 1; \\
\\ \Vert \xi \Vert - \frac{1}{2}  & \ \ \ {\rm if} \ \ \ \Vert \xi
\Vert \geq 1.
\end{array}
\right.
$$
The evolution problem for Lagrangian $f$, which does not include the nonparametric area integrand, but includes the plasticity functional was studied in \cite{HardtZhou}.

 A different type of example is the total variation flow, see for instance \cite{ACMBook}, and its anisotropic variant \cite{Moll}. In the classical total variation flow, the Lagrangian is given by the formula $f(x,\xi) = \vert \xi \vert$, and in particular it is not differentiable at $0$. Also, the corresponding $1$-Laplacian operator
\begin{equation}
\Delta_1 u := - \mathrm{div} \bigg(\frac{Du}{|Du|} \bigg)
\end{equation}
is strongly degenerate. For this reason, the techniques used for the study of the total variation flow are entirely different from the regular case, for instance compare the definitions of solutions to the total variation flow and for gradient flows for differentiable functionals of linear growth given in \cite[Chapter 2]{ACMBook} and \cite[Chapter 6]{ACMBook}.

Our main goal is to apply a duality-based method to provide a unified framework for the study of the total variation flow and other linear growth functionals. In the case of a general linear growth functional, the classical results in \cite{ACMBook} (see also \cite{ACMIbero}, \cite{ACM4:01} and \cite{ACMJEE}) require that the Lagrangian $f$ is differentiable at zero in the second variable, which is clearly not satisfied for the total variation flow. We give a general definition of solutions, which agrees with the known notions of solutions for the classical problems, and prove existence and uniqueness of solutions. The method we propose also reduces the assumptions needed and gives a more streamlined and simplified proof.

\medskip
Throughout the whole paper, we assume that $\Omega$ is a bounded Lipschitz domain in $\mathbb{R}^N$ and $f \in C(\overline{\Omega} \times \mathbb{R}^N)$ is the Lagrangian which is convex in the second variable and linearly coercive. We work with only two assumptions regarding the Lagrangian:

{\flushleft (A1)\labeltext{(A1)}{A1}} There exists $M > 0$ such that
\begin{equation}\label{eq:LIN}
|f(x,\xi)| \leq M(1 + |\xi|) \qquad \mbox{for all } (x,\xi) \in \Omega \times \mathbb{R}^N.
\end{equation}

{\flushleft (A2)\labeltext{(A2)}{A2}} The following limit exists:
$$ f^0(x,\xi) = \lim_{t \rightarrow 0^+} t f(x,\xi / t)$$
and it defines a function which is symmetric in $\xi$ and jointly continuous in $(x,\xi)$.

These conditions are similar to the assumptions (Lin) and (Con) given in \cite{BS} regarding optimality conditions for solutions to minimisation problems involving linear growth functionals. They are, however, significantly weaker than the assumptions used typically in the study of gradient flows of functionals with linear growth, where one either assumes a special form of the functional (e.g. for the total variation flow or its anisotropic version) or $C^1$ regularity of $f$. Let us compare our assumptions to the assumptions (H1)-(H5) given in \cite[Chapter 6]{ACMBook}: we assume condition (H1) except for the differentiability of $f$; we assume condition (H2); and we do not assume conditions (H3)-(H5). We also only require the boundary of the domain to be Lipschitz. In particular, our method not only generalises known results to new situations, but also provides a uniform framework which covers both the case of the (anisotropic) total variation flow and the case of functionals of class $C^1$.

Under the assumptions \ref{A1}-\ref{A2}, we study the Neumann and Dirichlet problems for gradient flows of general functionals of linear growth (we also remark on the Cauchy problem on the whole space). The Dirichlet problem for the total variation flow was studied extensively in the literature, see for instance \cite{ABCM2} and \cite{ACMBook}, and its anisotropic case was considered in \cite{Moll} and \cite{CFM}. The Dirichlet problem for general linear growth functionals, under more restrictive assumptions which include differentiability of $f$, was also studied in \cite{ACMBook}. On the other hand, the Neumann problem was so far only studied in very specific cases, such as the isotropic total variation flow (see \cite{ACMBook,GM2021-3}), and are entirely new in the general setting. We give a definition of weak solutions in the general case, prove existence and a characterisation of weak solutions, and highlight how the general definition looks like when specified to the important special cases.

The structure of the paper is as follows. In Section \ref{Preli}, we recall the necessary definitions and results concerning BV functions, Anzellotti pairings, and functions of a measure. We also prove Proposition \ref{prop:resultfrombs}, which is a crucial estimate on the Anzellotti pairing in terms of the asymptotic function of $f$. Then, in Section \ref{Neumann} we study the Neumann problem; we characterise the subdifferential of the functional associated to the gradient flow and use it to obtain a characterisation of weak solutions. In Section \ref{sec:dirichlet}, we make an analogous reasoning for the Dirichlet problem. In both cases, we complement the reasoning with several examples, and we comment how to adapt our technique for the Cauchy problem on the whole space.

\section{Preliminaries}\label{Preli}

\subsection{BV functions and Anzellotti pairings}

Due to the linear growth condition on the Lagrangian, the natural energy
space to study the problem is the space of functions of bounded variation. Let us recall several facts concerning functions of bounded variation (for further information we refer to  \cite{AFP}, \cite{EG} or \cite{Ziemer}). Throughout the whole paper, we assume that $\Omega \subset \R^N$ is an open bounded set with Lipschitz boundary.

\begin{definition}{\rm
A function $u \in L^1(\Omega)$ whose partial derivatives in the sense
of distributions are measures with finite total variation in $\Omega$
is called a function of bounded variation.
The space of such functions will be denoted by $BV(\Omega)$. In other words,
$u \in BV(\Omega)$ if and only if there exist Radon measures
$\mu_{1},\ldots,\mu_{N}$ defined in $\Omega$ with finite total mass
in $\Omega$ and
\begin{equation}
\int_{\Omega} u \, D_{i} \varphi \, dx = -\int_{\Omega} \varphi \, d\mu_{i}
\end{equation}
for all $\varphi \in C^{\infty}_{0}(\Omega)$,
$i=1,\ldots,N$. Thus, the distributional gradient of $u$ (denoted $Du$) is a vector valued measure with finite total variation
\begin{equation}
\label{deftv}
\vert  Du \vert (\Omega) = \sup \bigg\{ \int_{\Omega} \, u \, \mathrm{div} \, \varphi \, dx: \,\, \varphi \in C^{\infty}_{0}(\Omega; \R^N), \,\, |\varphi(x)| \leq 1
\, \, \hbox{for $x \in \Omega$} \bigg\}.
\end{equation}
The space $BV(\Omega)$ is endowed with the norm
\begin{equation}
\| u \|_{BV(\Omega)} = \| u \|_{L^1(\Omega)} +
\vert  Du \vert (\Omega).
\end{equation}
}
\end{definition}

For $u \in BV(\Omega)$, the distributional gradient $Du$ is a Radon measure that decomposes
into its absolutely continuous and singular parts $Du = D^a u + D^s u$. Then $D^a u = \nabla u \ \mathcal{L}^N$, where $\nabla u$ is the Radon-Nikodym derivative of
the measure $Du$ with respect to the Lebesgue measure $\mathcal{L}^N$. There is also the polar decomposition $D^s u = \overrightarrow{D^s u} \vert D^s u \vert$, where $\vert D^s u \vert$ is the total variation measure of $D^s u$. We denote by $\nu_u:= \frac{d Du}{d\vert Du \vert}$ the Radon-Nikodym derivative of the measure $Du$ with respect to the measure $\vert Du \vert$.

We shall need several results from \cite{Anzellotti1} (see also \cite{KohnTeman}). Following \cite{Anzellotti1}, let
\begin{equation}
\label{xw}
X_p(\Omega) = \{ \z \in L^{\infty}(\Omega; \R^N):
\divi(\z)\in L^p(\Omega) \}.
\end{equation}

\begin{definition}{\rm
For $\z \in X_p(\Omega)$ and $u \in BV(\Omega) \cap L^{p'}(\Omega)$,
define the functional $(\z,Du): C^{\infty}_{0}(\Omega) \rightarrow \R$
by the formula
\begin{equation}\label{Anzel1}
\langle (\z,Du),\varphi \rangle = - \int_{\Omega} u \, \varphi \, \divi(\z) \, dx -
\int_{\Omega} u \, \z \cdot \nabla \varphi \, dx.
\end{equation}
}

\end{definition}

 The following result collects some of the most important properties of the pairing $(\z, Du)$, formally defined only as a distribution on $\Omega$.

\begin{proposition}
The distribution $(\z,Du)$ is a Radon measure in $\Omega$. Moreover,
\begin{equation}\label{eq:absolutecontinuity}
\bigg\vert \int_{B} (\z,Du) \bigg\vert \leq \int_{B} |(\z,Du)| \leq \| \z
\|_{\infty}
\int_{B} \vert Du \vert
\end{equation}
for any Borel set $B \subseteq \Omega$. In particular, $(\z,Du)$ is absolutely continuous with respect to $\vert Du \vert$. Furthermore,
\begin{equation}
\int_{\Omega} (\z,Dw) = \int_{\Omega} \z \cdot \nabla w \, dx \ \ \ \ \ \forall \ w
\in W^{1,1}(\Omega) \cap L^{\infty}(\Omega)
\end{equation}
so $(\z,Du)$ agrees on Sobolev functions with the dot product of $\z$ and $\nabla u$.
\end{proposition}

By \eqref{eq:absolutecontinuity}, the measure $(\z,Du)$ has a Radon-Nikodym derivative with respect to $|Du|$
$$\theta (\z,Du, \cdot):= \frac{d[(\z, Du)]}{d|Du|}$$
which is a $\vert Du \vert$-measurable function from $\Omega$ to $\R$ such that
\begin{equation}
\int_{B} (\z,Du) = \int_{B} \theta (\z,Du,x) \vert Du \vert
\end{equation}
for any Borel set $B \subseteq \Omega$. We also have that
\begin{equation}
\| \theta (\z, Du, \cdot) \|_{L^{\infty}(\Omega,\vert Du \vert )} \leq
\| \z \|_{L^{\infty}(\Omega;\R^N)}.
\end{equation}
In \cite{Anzellotti1}, a weak trace on $\partial \Omega$ of the normal component
of \ $\z \in X_p(\Omega)$ is defined. Concretely, it is proved that there exists a linear operator \
$\gamma : X(\Omega) \rightarrow L^{\infty}(\partial \Omega)$ such that
$$\Vert \gamma(\z) \Vert_{\infty} \leq \Vert \z \Vert_{\infty}$$
$$\gamma(\z) (x) = \z(x) \cdot \nu_\Omega(x) \ \ \ \ {\rm for \ all} \ x \in
\partial
\Omega \ \ {\rm if} \
\ \z \in C^1(\overline{\Omega}, \R^N),$$
being $\nu_\Omega$ the unit outward normal on $\partial \Omega$. We shall denote \ $\gamma (\z)(x)$ by $[\z, \nu_\Omega](x)$. Moreover, the following {\it Green's formula}, relating the function $[\z, \nu_\Omega]$ and the measure $(\z, Dw)$, was proved in the same paper.

\begin{theorem}
For all $\z \in X_p(\Omega)$ and $u \in BV(\Omega) \cap L^{p'}(\Omega)$, we have
\begin{equation}\label{Green}
\int_{\Omega} u \ \divi (\z) \, dx + \int_{\Omega} (\z, Du) = \int_{\partial \Omega} u \, [\z, \nu_\Omega] \, d\mathcal{H}^{N-1}.
\end{equation}
\end{theorem}

Moreover, as a consequence of \cite[Proposition 2.2]{Anzellotti4} and \cite[Theorem 3.6]{Anzellotti4}  (see also \cite{Anzellotti3} and \cite{CCCM}), we have the following result.

\begin{theorem}
Suppose that $u \in BV(\Omega)$ and $\z \in X_1(\Omega)$. Then,
\begin{equation}\label{Anze1}
\frac{d[(\z, Du)]}{d|Du|}(x) = \lim_{\rho \to 0^+} \lim_{r \to 0^+} \frac{1}{2r \omega_{N-1} \rho^{N-1}} \int_{C_{r,\rho}(x, \nu_u(x))} \z(y) \cdot \nu_u(x) \, dy
\end{equation}
for $|Du|$-a.e. $x \in \Omega$. Here,
$$C_{r,\rho}(x, \alpha):= \{ \xi \in \R^N \ : \ \vert (\xi - x) \cdot \alpha \vert < r, \  \vert (\xi - x) - [(\xi - x)\cdot \alpha ] \alpha \vert < \rho \}, \quad \alpha \in S^{N-1}.$$
 Furthermore,
\begin{equation}\label{Anze2}
[\z, \nu_\Omega](x) = \lim_{\rho \to 0^+} \lim_{r \to 0^+} \frac{1}{2r \omega_{N-1} \rho^{N-1}} \int_{C_{r,\rho}(x, \nu_\Omega(x))} \z(y) \cdot \nu_\Omega(x) \, dy
\end{equation}
for $\mathcal{H}^{N-1}$-a.e. $x \in \partial\Omega$.
\end{theorem}

By writing
$$\z \cdot D^s u:= (\z, Du) - (\z \cdot \nabla u) \, d \mathcal{L}^N,$$
we see that $\z \cdot D^s u$ is a bounded measure. Furthermore, in
\cite{ACMBook}   it is proved
that $\z \cdot D^s u$ is absolutely continuous with respect to
$\vert D^s u \vert$ (and, thus, it is a singular
measure with respect to $\mathcal{L}^N$), and
\begin{equation}\label{singular}
\vert \z \cdot D^s u \vert \leq \Vert \z \Vert_{\infty} \vert D^s u
\vert.
\end{equation}
As a consequence of \cite[Theorem 2.4]{Anzellotti1}, we have:
\begin{equation}\label{CampCont}
{\rm If} \ \ \ \z \in X_p(\Omega) \cap C(\Omega; \R^N), \ \ {\rm then} \ \ \  \z
\cdot D^s u = (\z \cdot \overrightarrow{D^s u}) \, d \vert D^s u \vert.
\end{equation}

\subsection{Function of a measure}\label{sec:functionofmeasure}

In order to consider the relaxed energy we recall the definition of a function of a measure (see for instance \cite{Anzellotti2} or \cite{DemengelTeman1}). Let $f: \Omega \times \R^N \rightarrow \R$ be a Carath\'eodory function such that
\begin{equation}\label{LGRWTH}
\vert f(x, \xi) \vert \leq M (1 + \Vert \xi \Vert) \ \ \ \ \ \ \forall \ (x,
\xi) \in \Omega \times \R^N,
\end{equation}
for some constant $M \geq 0$. Furthermore, we assume that $f$ possesses
an asymptotic function, i.e.  for almost all $x \in \Omega$ there exists the finite limit
\begin{equation}\label{Asimptfunct}
\lim_{t \to 0^+} tf \bigg( x, \frac{\xi}{t} \bigg) = f^0(x, \xi).
\end{equation}
It is clear that the function \ $f^0(x, \xi)$ is positively homogeneous of degree one in $\xi$, i.e.
$$f^0(x, s \xi) = s f^0(x, \xi) \ \ \ \ \ {\rm for \ all} \ x, \xi \ {\rm
and} \ s > 0.$$
Throughout the paper, we will assume that $f^0$ is symmetric, i.e.,
\begin{equation}\label{cpndI}
f^0(x, -\xi) = f^0(x, \xi)\quad \hbox{for all  $\xi \in \R^N$ and all $x \in \overline{\Omega}$}.
\end{equation}
It is easy to see that
\begin{equation}\label{good01}
\hbox{If} \ f(x, \eta) - f(x,\xi) \geq z \cdot (\eta - \xi),\quad \hbox{then} \ \ \ \ f^0(x, \eta) - f^0(x,\xi) \geq z \cdot (\eta - \xi).
\end{equation}
Let us recall that by definition of the subgradient,
$$z \in \partial_\xi f(x, \xi) \iff f(x, \eta) - f(x,\xi) \geq z \cdot (\eta - \xi)\quad \forall \eta \in \R^N.$$
Then, \eqref{good01} means that
$$z \in \partial_\xi f(x, \xi) \quad  \Rightarrow \quad z \in \partial_\xi f^0(x, \xi).$$
We denote by ${\mathcal M}(\Omega; \R^N)$ the set of all $\R^N$-valued
bounded Radon measures on $\Omega$. Given $\mu \in {\mathcal M}(\Omega; \R^N)$, we consider its Lebesgue decomposition
$$\mu = \mu^a + \mu^s,$$
where $\mu^a$ is the absolutely continuous part of $\mu$ with respect to the Lebesgue measure $\mathcal{L}^N$, and $\mu^s$ is singular with respect to  $\mathcal{L}^N$. We denote by $\mu^a(x)$ the density of the measure $\mu^a$ with respect to $\mathcal{L}^N$ and by $(d \mu^s / d \vert \mu \vert^s)(x)$ the density of $\mu^s$ with respect to $\vert \mu \vert^s$.

Given $\mu \in {\mathcal M}(\Omega; \R^N)$, we define $\tilde{\mu} \in {\mathcal M}(\Omega; \R^{N+1})$ by
$$\tilde{\mu}(B):= \big( \mu(B), \mathcal{L}^N(B) \big),$$
for every Borel set $B \subset \R^N$. Then, we have
$$\tilde{\mu} = \tilde{\mu}^a + \tilde{\mu}^s = \tilde{\mu}^a(x) \mathcal{L}^{N+1} + \tilde{\mu}^s = (\mu^a(x), \1_\Omega) \mathcal{L}^{N+1} +  (\mu^s, 0).$$
Hence, we have
$$\vert \tilde{\mu}^s \vert = \vert \mu^s \vert, \ \ \ \ \ \frac{d \tilde{\mu}^s}{d \vert \tilde{\mu}^s \vert} = \bigg(
\frac{d \mu^s}{d \vert \mu^s \vert}, 0 \bigg) \ \ \ \vert \mu^s \vert-{\rm a.e.}$$
For $\mu \in {\mathcal M}(\Omega; \R^N)$ and $f$ satisfying the above conditions, we define the measure $f(x, \mu)$ on $\Omega$ as
\begin{equation}\label{defmeasure}
\int_B f(x, \mu) := \int_B f(x, \mu^a(x)) \, dx + \int_B f^0 \bigg(x,
\frac{d \mu^s}{d \vert \mu \vert^s}(x) \bigg) \, d\vert \mu \vert^s
\end{equation}
for all Borel set $B \subset \Omega$. In formula \eqref{defmeasure} we may write $(d \mu / d \vert \mu \vert)(x)$ instead of $(d \mu^s / d \vert \mu \vert^s)(x)$, because the two functions are equal $\vert \mu \vert^s$-a.e.

Another way of looking at the measure $f(x, \mu)$ is the following. Let us
consider the function $\tilde{f}: \Omega \times \R^N \times [0, + \infty) \rightarrow \R$ defined as
\begin{equation}\label{funcrecesion}
\tilde{f}(x, \xi, t):= \left\{
\begin{array}{ll} f \displaystyle \bigg(x, \frac{\xi}{t} \bigg)t &
\hspace{0.5cm}\hbox{if } t > 0; \\ \\
f^0(x, \xi)  &
\hspace{0.5cm}\hbox{if } t = 0.
\end{array}
\right.
\end{equation}
As it is proved in \cite{Anzellotti2}, if $f$ is a Carath\'eodory function
satisfying \eqref{LGRWTH}, then one has
\begin{equation}\label{functmeasure}
\int_B f(x, \mu) = \int_B \tilde{f} \bigg(x, \frac{d \mu}{d \alpha}(x),
\frac{d \mathcal{L}^N}{d \alpha} (x) \bigg) \, d \alpha,
\end{equation}
where $\alpha$ is any positive Borel measure such that $\vert \mu \vert +
\mathcal{L}^N \ll \alpha$.

Now, we will apply the general theory of a function of a measure to the particular case $\mu = Du$, where $u \in BV(\Omega)$. Let $f$ be a function satisfying \eqref{LGRWTH}. Then for every $u \in
BV(\Omega)$ we have the
measure $f(x, Du)$ defined by
$$\int_B f(x, Du) = \int_B f(x, \nabla u(x)) \, dx + \int_B f^0(x,\overrightarrow{D^s u}(x)) \, d\vert D^s u \vert$$
for all Borel set $B \subset \Omega$. Since we assume that $\Omega$ has Lipschitz boundary, and therefore $f(x, \xi)$ is defined also for $x \in \partial \Omega$, we may consider the functional $G$ in $BV(\Omega)$ defined
by
\begin{equation}\label{Functional1}
G(u):= \int_{\Omega} f(x, Du) + \int_{\partial \Omega} f^0 \big(x,\nu_\Omega(x) [\varphi (x) - u(x) ] \big) \, d\mathcal{H}^{N-1},
\end{equation}
where $\varphi \in L^1(\partial \Omega)$ is a given function and $\nu_\Omega$
is the outer unit normal to $\partial \Omega$. It is proved in \cite{Anzellotti2} that, if $\tilde{f}(x, \xi, t)$ is continuous on $\overline{\Omega} \times \R^N \times [0, + \infty)$ and convex in $(\xi,t)$ for each fixed
$x \in \overline{\Omega}$, then $G$ is the greatest functional on $BV(\Omega)$
which is lower-semicontinuous with respect to the $L^1(\Omega)$-convergence and
satisfies $G(u) \leq \displaystyle\int_{\Omega} f(x, \nabla u(x)) \, dx$ for all
functions $u \in C^1(\Omega) \cap W^{1,1}(\Omega)$
with $u = \varphi$ on $\partial \Omega$.  Let us remark that the above result is proved in \cite[Lemma 6.3]{ACMBook} assuming only that $\tilde{f}(x, \xi, t)$ is lower semicontinuous on $\overline{\Omega} \times \R^N \times [0, + \infty)$.

Finally, we recall a version of the Reshetnyak continuity theorem suitable for our purposes, proved in \cite[Theorem 3.10]{BS}. We recall it in full generality, when $\Omega$ is allowed to be an unbounded subset of $\R^N$; since we work in bounded domains, in this paper it will be enough to apply it in the case $\rho \equiv 1$.

\begin{theorem}\label{thm:reshetnyak}
Suppose that $\Omega \subset \mathbb{R}^N$ is an open set such that $\partial\Omega$ has zero Lebesgue measure. Moreover, assume that $f: \Omega \times \mathbb{R}^N \rightarrow \mathbb{R}$ satisfies \ref{A2} and $|f(x,\xi)| \leq \psi(x) + M|\xi|$ with $\psi \in L^1(\Omega)$ (for bounded domains this condition reduces to \ref{A1}). Let $\mu_k$ be a sequence of finite vector-valued Radon measures on $\overline{\Omega}$ which weakly* converges to $\mu$. If there holds
$$ \lim_{k \rightarrow \infty} |(\rho \mathcal{L}^n, \mu_k)|(\overline{\Omega}) = |(\rho \mathcal{L}^n, \mu)|(\overline{\Omega})$$
for some positive $\rho \in L^1(\Omega)$ bounded away from zero on every bounded subset of $\Omega$, we have
\begin{align}\label{eq:reshetnyak}
\lim_{k \rightarrow \infty} \bigg[ \int_\Omega f \bigg(\cdot, \frac{d\mu^a_k}{d\mathcal{L}^n} \bigg) \, dx + \int_{\overline{\Omega}} &f^0\left(\cdot, \frac{d\mu^s_k}{d|\mu^s_k|}\right) d|\mu^s_k| \bigg] \\
&= \int_\Omega f\left(\cdot, \frac{d\mu^a}{d\mathcal{L}^n}\right) \, dx + \int_{\overline{\Omega}} f^0\left(\cdot, \frac{d\mu^s}{d|\mu^s|}\right) d|\mu^s|.
\end{align}
\end{theorem}

\subsection{Estimates in terms of $f^0$} In this subsection, we show a pointwise estimate of the Radon-Nikodym derivative of $(\z, Du)$  with respect to $|Du|$ and related objects in terms of the asymptotic function $f^0$ of the functional $f$. Recall that the {\it conjugate function} $f^*: \Omega \times \R^N \rightarrow (-\infty,+\infty]$ of $f$ (with respect to the $\xi$-variable) is defined as
$$f^*(x, \xi^*):= \sup_{\xi\in \R^N} [\xi^* \cdot \xi - f(x, \xi)].$$

\begin{proposition}\label{prop:resultfrombs}
Suppose that $\Omega \subset \R^N$ has Lipschitz boundary and that $f$ satisfies \ref{A1} and \ref{A2}. Then, for all $u \in BV(\Omega) \cap L^{p'}(\Omega)$ and $\z \in X_p(\Omega)$ such that $f^*(\cdot,\z) < \infty$ holds $\mathcal{L}^N$-a.e. on $\Omega$, we have
\begin{equation}\label{good0}
f^0\left(\cdot, \frac{dDu}{d|Du|}\right) \geq \frac{d[(\z, Du)]}{d|Du|} \qquad |Du|-\mbox{a.e. in } \Omega,
\end{equation}
\begin{equation}\label{good1}
f^0\left(\cdot, \frac{dD^s u}{d|D^s u|}\right) \geq \frac{d[\z \cdot D^s u]}{d|D^s u|} \qquad |D^s u|-\mbox{a.e. in } \Omega,
\end{equation}
and
\begin{equation}\label{good2}
f^0(\cdot, \nu_\Omega ) \vert u \vert \geq - [\z,\nu_\Omega] u \qquad \mathcal{H}^{N-1}-\mbox{a.e. on } \partial\Omega.
\end{equation}
\end{proposition}

\begin{proof}
By the definition of the conjugate function, we have
$$t f\left(x, \frac{\xi}{t}\right) \geq \z(x) \cdot \xi - t  f^*(x, \z(x)) \quad \hbox{for all} \ (x,\xi) \in \Omega \times \R^N \ \hbox{and all} \ t>0.$$
Let $t \rightarrow 0$. Then, by assumption \ref{A2}, we get
$$f^0(x, \xi) \geq \z(x) \cdot \xi \quad \hbox{for all} \ \xi \in \R^N,$$
whenever $x \in \Omega$ is such that  $f^*(x, \z(x))$ is finite.  By the assumption that $f^*(\cdot,\mathbf{z}) < \infty$ $\mathcal{L}^N$-a.e. in $\Omega$, we have
\begin{equation}\label{eq:inclusionae}
\z(x) \in \partial_\xi f^0(x, 0) \quad \hbox{for $\mathcal{L}^N$-a.e. } x\in \Omega.
\end{equation}

Now, equation \eqref{Anze1} implies that for $|Du|-\mbox{a.e.} \ x_0 \in \Omega$ we have
\begin{equation}\label{Anze1e1}
\frac{d[(\z, Du)]}{d|Du|}(x_0) = \lim_{\rho \to 0^+} \lim_{r \to 0^+} \frac{1}{2r \omega_{N-1} \rho^{N-1}} \int_{C_{r,\rho}(x_0, \nu_u(x))} \z(y) \cdot \nu_u(x_0) \, dy.
\end{equation}
Fix an $x_0 \in \Omega$ satisfying \eqref{Anze1e1}. By equation \eqref{eq:inclusionae} and \cite[Lemma 5.4]{BS}, for any $\varepsilon >0$ there exists $\delta >0$ such that
$$\z(x) \in \mathcal{N}_\varepsilon (\partial_\xi f^0(x_0, 0)) \quad \hbox{for $\mathcal{L}^N$-almost} \ x\in \Omega \ \hbox{such that} \ \vert x - x_0 \vert < \delta,$$
where $\mathcal{N}_\varepsilon(\partial_\xi f^0(x_0, 0))$ denotes the $\varepsilon$-neighbourhood of the set $\partial_\xi f^0(x_0, 0)$. Thus, by \eqref{Anze1e1} there exists $\xi_0 \in \overline{\mathcal{N}_\varepsilon (\partial_\xi f^0(x_0, 0))}$ such that
$$\frac{d[(\z, Du)]}{d|Du|}(x_0) = \xi_0 \nu_u(x_0).$$
Consequently, we can find a subgradient $\xi^* \in \partial_\xi f^0(x_0, 0)$ with $\vert \xi_0 - \xi^* \vert \leq \varepsilon.$ Hence
$$\frac{d[(\z, Du)]}{d|Du|}(x_0) = \xi_0 \nu_u(x_0) \leq \xi^* \nu_u(x_0) + \varepsilon \vert \nu_u(x_0) \vert \leq f^0(x_0,\nu_u(x_0)) + \varepsilon \vert \nu_u(x_0) \vert. $$
We conclude the proof of \eqref{good0} by letting $\varepsilon \rightarrow 0$. The property \eqref{good1} is a consequence of \eqref{good0}, having in mind the decomposition of the measure into its absolutely continuous part with respect to the Lebesgue measure and its singular part.

Finally, let us prove \eqref{good2}. By \eqref{Anze2}, for $\mathcal{H}^{N-1}-\mbox{a.e.} \ x_0 \in \partial\Omega$ we have
\begin{equation}\label{Anze2e2}
[\z, \nu_\Omega](x_0) = \lim_{\rho \to 0^+} \lim_{r \to 0^+} \frac{1}{2r \omega_{N-1} \rho^{N-1}} \int_{C_{r,\rho}(x_0, \nu_\Omega(_0))} \z(y) \cdot \nu_\Omega(x_0) \, dy.
\end{equation}
Fix an $x_0 \in \partial\Omega$ satisfying \eqref{Anze2e2}. Working as in the first part of the proof, we obtain that given $\varepsilon >0$, there exists $\delta >0$ such that
$$\z(x) \in \mathcal{N}_\varepsilon (\partial_\xi f^0(x_0, 0)) \quad \hbox{for $\mathcal{H}^{N-1}$-almost} \ x\in \partial \Omega, \ \hbox{such that} \ \vert x - x_0 \vert < \delta,$$
so by \eqref{Anze2e2} there exists $\xi_0 \in \overline{\mathcal{N}_\varepsilon (\partial_\xi f^0(x_0, 0))}$ such that
$$[\z, \nu_\Omega](x_0) = \xi_0 \nu_\Omega(x_0).$$
Consequently, we can find a subgradient $\xi^* \in \partial_\xi f^0(x_0, 0)$ with $\vert \xi_0 - \xi^* \vert \leq \varepsilon.$ Hence,
\begin{align}
-[\z, \nu_\Omega](x_0) u(x_0) &= -\xi_0 \nu_\Omega(x_0) u(x_0) = (\xi^*- \xi_0)\nu_\Omega(x_0) u(x_0) - \xi^* \nu_\Omega(x_0) u(x_0) \\
&\leq \varepsilon \vert u(x_0) \vert - \xi^* \nu_\Omega(x_0) u(x_0) \leq \varepsilon \vert u(x_0) \vert + f^0(x_0, \nu_\Omega(x_0)) \vert u(x_0) \vert,
\end{align}
where the last inequality is a consequence of $\xi^* \in \partial_\xi f^0(x_0, 0)$ and \eqref{cpndI}. We conclude the proof of~inequality \eqref{good2} by letting $\varepsilon \rightarrow 0$.
\end{proof}

\begin{remark} {\rm
The main motivation for Proposition \ref{prop:resultfrombs} and its proof is \cite[Corollary 5.3]{BS}, in which a similar result is shown for the up-to-boundary Anzellotti pairing introduced in \cite{BS}. However, let us point out that we could not apply this result directly, because the up-to-boundary pairing from \cite{BS} is only comparable with the standard Anzellotti pairing in the case ${\rm div}(\z) = 0$. Therefore, we decided to give a direct proof using the results of Anzellotti concerning the pointwise behaviour of $(\mathbf{z},Du)$ and $[\mathbf{z},\nu_\Omega]$.
}
\end{remark}

\section{The Neumann problem}\label{Neumann}

Recall that throughout the whole paper we assume that $\Omega$ is an open bounded subset of $\R^N$ with Lipschitz boundary, where $N \geq 2$. Moreover, the Lagrangian  $f \in C(\overline{\Omega} \times \mathbb{R}^N)$ is convex in the variable $\xi$ and satisfies conditions \ref{A1} and \ref{A2}. In this section, we study the Neumann problem
\begin{equation}\label{NeumannProb}
\left\{ \begin{array}{lll} u_t (t,x) = {\rm div} (\partial_\xi f(x, Du(t,x)))   \quad &\hbox{in} \ \ (0, T) \times \Omega; \\[5pt] \frac{\partial u}{\partial \eta} = 0 \quad &\hbox{on} \ \ (0, T) \times  \partial\Omega; \\[5pt] u(0,x) = u_0(x) \quad & \hbox{in} \ \  \Omega, \end{array} \right.
\end{equation}
where $u_0 \in L^2(\Omega)$. In the case when $f \in C^1(\overline{\Omega} \times \R^N)$, the subdifferential is single-valued and the problem reduced to the classical case
\begin{equation}\label{eq:singlevaluedcase}
\left\{ \begin{array}{lll} u_t (t,x) = {\rm div}\,  \a(x, Du(t,x))   \quad &\hbox{in} \ \ (0, T) \times \Omega; \\[5pt] \a(x, Du(t,x)) \cdot \nu_\Omega = 0 \quad &\hbox{on} \ \ (0, T) \times  \partial\Omega; \\[5pt] u(0,x) = u_0(x) \quad & \hbox{in} \ \  \Omega, \end{array} \right.
\end{equation}
where $\a(x, \xi)  = \partial_\xi f(x, \xi)$. Another important special case is when $f(x,\xi) = |\xi|$, when we recover the total variation flow
\begin{equation}\label{eq:totalvariationflow}
\left\{ \begin{array}{lll} u_t (t,x) = {\rm div} \bigg(\frac{Du(t,x)}{|Du(t,x)|} \bigg)  \quad &\hbox{in} \ \ (0, T) \times \Omega; \\[5pt] \frac{Du(t,x)}{|Du(t,x)|} \cdot \nu_\Omega = 0 \quad &\hbox{on} \ \ (0, T) \times  \partial\Omega; \\[5pt] u(0,x) = u_0(x) \quad & \hbox{in} \ \  \Omega. \end{array} \right.
\end{equation}
For a comprehensive study of problem \eqref{eq:totalvariationflow} see \cite{ABCM1} or \cite{ACMBook}. Due to the fact that in general the subdifferential $\partial_\xi f(x, Du(t,x))$ might not be single-valued, we need to give a suitable definition of weak solutions, and our approach is closer to the one used for problem \eqref{eq:totalvariationflow} than the one for problem \eqref{eq:singlevaluedcase}. In particular, our definition will be based on the existence of a suitable vector field $\z \in X_2(\Omega)$, which replaces the possibly multivalued object $\partial_\xi f(x, Du(t,x))$.

Consider the energy functional $\mathcal{F}_{\mathcal{N}} : L^2(\Omega) \rightarrow [0, + \infty]$ associated with problem \eqref{NeumannProb} and defined by
\begin{equation}
\mathcal{F}_{\mathcal{N}}(u):= \left\{ \begin{array}{ll} \displaystyle\int_\Omega f(x,Du)  \quad &\hbox{if} \ u \in BV(\Omega) \cap L^2(\Omega); \\ \\ + \infty \quad &\hbox{if} \ u \in  L^2(\Omega) \setminus BV(\Omega).
\end{array}\right.
\end{equation}
We have that $\mathcal{F}_{\mathcal{N}}$ is convex and lower semicontinuous with respect to the $L^2(\Omega)$-convergence. Then, by the theory of maximal monotone operators (see \cite{Brezis}) there is a unique strong solution of the abstract Cauchy problem
\begin{equation}
\left\{ \begin{array}{ll} u'(t) + \partial \mathcal{F}_{\mathcal{N}}(u(t)) \ni 0, \quad t \in [0,T] \\[5pt] u(0) = u_0. \end{array}\right.
\end{equation}

To characterise the subdifferential of $\mathcal{F}_{\mathcal{N}}$, we introduce the following operator.

\begin{definition}{\rm We say that $(u,v) \in \mathcal{A_N}$ if and only if $u, v \in L^2(\Omega)$, $u \in BV(\Omega)$ and  there exists a vector field $\z \in X_2(\Omega)$  such that the following conditions hold:
\begin{equation}\label{e1Def}\z \in \partial_\xi f(x,\nabla u) \quad \mathcal{L}^N-\hbox{a.e. in } \Omega;\end{equation}
\begin{equation}\label{e2Def}  -\mathrm{div}(\z) = v \quad \hbox{in} \ \mathcal{D}^\prime(\Omega);\end{equation}
\begin{equation}\label{e3Def} \z \cdot D^s u = f^0(x,D^s u) = f^0(x,\overrightarrow{D^s u}) \vert
D^s u \vert \quad \hbox{as measures};\end{equation}
\begin{equation}\label{e4Def}[\z, \nu_\Omega] =0 \qquad  \mathcal{H}^{N-1}-\mbox{a.e. on } \partial \Omega.\end{equation}
}
\end{definition}

Note that this definition covers both the case of the total variation flow and convex differentiable functionals with linear growth; in the case of the total variation flow, $f(x, \xi) = \Vert \xi \Vert$, we have $f = f^0$ and the subdifferential condition means that $\| \z\|_\infty \leq 1$  (and similarly for the anisotropic TV flow studied in \cite{Moll} on the whole $\R^N$). In the differentiable case, we have $\z = \nabla_\xi f(x,\xi)$ a.e. We will revisit these examples at the end of this Section.

Our main goal in this Section is to prove that $\mathcal{A}_{\mathcal{N}}$ coincides with the subdifferential of $\mathcal{F}_{\mathcal{N}}$ and study some consequences of this result. To get this characterisation, we need to use the version of the Fenchel-Rockafellar duality Theorem given in \cite[Remark III.4.2]{EkelandTemam}.

Let $U,V$ be two Banach spaces and let $A: U \rightarrow V$ be a continuous linear operator. Denote by $A^*: V^* \rightarrow U^*$ its dual. Then, if the primal problem is of the form
\begin{equation}\tag{P}\label{eq:primal}
\inf_{u \in U} \bigg\{ E(Au) + G(u) \bigg\},
\end{equation}
then the dual problem is defined as the maximisation problem
\begin{equation}\tag{P*}\label{eq:dual}
\sup_{p^* \in V^*} \bigg\{ - E^*(-p^*) - G^*(A^* p^*) \bigg\},
\end{equation}
where $E^*$ and $G^*$ are the Legendre–Fenchel transformations (conjugate  functions) of $E$ and $G$ respectively, i.e.,
$$E^* (u^*):= \sup_{u \in U} \left\{ \langle u, u^* \rangle_{ U,U^*} - E(u) \right\}.$$

\begin{theorem}[Fenchel-Rockafellar Duality Theorem]\label{FRTh} Assume that $E$ and $G$ are proper, convex and lower semicontinuous. If there exists $u_0 \in U$ such that $E(A u_0) < \infty$, $G(u_0) < \infty$ and $E$ is continuous at $A u_0$, then
$$\inf \eqref{eq:primal} = \sup \eqref{eq:dual}$$
and the dual problem \eqref{eq:dual} admits at least one solution. Moreover, the optimality conditions between these two problems are given by
\begin{equation}\label{optimality}
A^* p^* \in \partial G(\overline{u}) \quad \mbox{and} \quad -p^* \in \partial E(A\overline{u})),
\end{equation}
where $\overline{u}$ is solution of \eqref{eq:primal} and $p^*$ is solution of \eqref{eq:dual}.
\end{theorem}

 In the case when there is no solution to the dual problem, instead of optimality conditions we have the $\varepsilon-$subdifferentiability property of minimising sequences, see \cite[Proposition V.1.2]{EkelandTemam}: for any minimising sequence $u_n$ for \eqref{eq:primal} and a maximiser $p^*$ of \eqref{eq:dual}, we have
\begin{equation}\label{eq:epsilonsubdiff11N}
0 \leq E(Au_n) + E^*(-p^*) - { \langle Au_n, -p^* \rangle_{V,V^*} } \leq \varepsilon_n
\end{equation}
\begin{equation}\label{eq:epsilonsubdiff21N}
0 \leq G(u_n) + G^*(A^* p^*) - \langle u_n, A^* p^* \rangle_{ U,U^*} \leq \varepsilon_n
\end{equation}
with $\varepsilon_n \rightarrow 0$.

We now show a simple technical lemma concerning the convex conjugate of a function defined on an intersection of two Banach spaces.

\begin{lemma}\label{ap.a.lem.dual}
Let $X,Y$ be two Banach spaces and suppose that $G: X \cap Y \rightarrow \mathbb{R}$ is a functional which satisfies the bound
\begin{equation}\label{eq:1}
G(u) \leq \ell(u)
\end{equation}
with $\ell: X \cap Y \rightarrow \mathbb{R}$ a convex function such that
\begin{equation}\label{eq:2}
|\ell(u)| \leq \ell_0(\| u \|_Y)
\end{equation}
for a nondecreasing function $\ell_0: [0,+\infty) \rightarrow \mathbb{R}$. Then, if $G^*(u^*) < \infty$, we have that $u^* \in Y^*$.
\end{lemma}

\begin{proof}
For any $ u \in X \cap Y$, set
$$ L_{u^*}(u) := \langle u, u^* \rangle_{X \cap Y, (X \cap Y)^*}. $$
It is clearly a linear functional. By the definition of the convex conjugate, we have
$$ L_{u^*}(u) \leq G^*(u^*) + G(u).$$
By the assumptions \eqref{eq:1} and \eqref{eq:2}, we have
$$ L_{u^*}(u) \leq G^*(u^*) + \ell(u) \leq G^*(u^*) + \ell_0(\| u \|_Y).$$
The same inequality holds for $-u$, so
$$ |L_{u^*}(u)| \leq G^*(u^*) + \leq \ell_0(\| u \|_Y).$$
Now, $L_{u^*}: X \cap Y \rightarrow \mathbb{R}$ is linear and $\ell$ is a convex function, so by the Hahn-Banach theorem (in the form presented e.g. in \cite{RS}) there exists a linear extension $L$ of $L_{u^*}$ to $Y$ such that $|L(u)| \le \ell(u).$ Hence, if $||u||_{Y}\le 1$, we have
$$|L(u)|\le G^*(A^* u^*) + \ell_0(1)$$
and consequently $L$ is continuous. Therefore, $L$ is represented by an element of~$Y^*$, hence $L_{u^*}$ is as well, and we conclude that $u^* \in Y^*$.
\end{proof}

We now proceed with the characterisation of the operator $\mathcal{A_N}$.

\begin{lemma}\label{lm:Neumann} Under the assumptions \ref{A1}-\ref{A2}, we have
\begin{equation}\label{e3.lm:Neumann}
\mathcal{A_N} \subset \partial \mathcal{F_N}.
\end{equation}
\end{lemma}

\begin{proof}  Let $(u,v) \in \mathcal{A_N}$ and $\z \in X_2(\Omega)$ satisfying \eqref{e1Def}-\eqref{e4Def}. Given $w \in L^2(\Omega) \cap BV(\Omega)$, by \eqref{e1Def} we have
\begin{equation}\label{e1.lm:Neumann}
f(x,\nabla w(x)) - f(x,\nabla u(x)) \geq \z(x) \cdot (\nabla w(x) - \nabla u(x)) \quad \quad  \mathcal{L}^N-\hbox{a.e. in } \Omega.
\end{equation}
Then, since $\z \in X_2(\Omega)$ satisfies \eqref{e1Def}-\eqref{e4Def}, applying Green formula \eqref{Green} and having in mind properties \eqref{good1} and \eqref{e1.lm:Neumann}, we have
\begin{align}
\int_\Omega (w - u) \, v \, dx &= - \int_\Omega \mathrm{div}(\z) (w - u) \, dx = \int_\Omega (\z, Dw) -  \int_\Omega (\z, Du) \\
&= \int_\Omega \z \cdot \nabla w \, dx + \int_\Omega \z \cdot D^s w - \int_\Omega \z \cdot \nabla u \, dx - \int_\Omega \z \cdot D^s u \\
&\leq \int_\Omega f(x, \nabla w) \, dx - \int_\Omega f(x, \nabla u) \, dx + \int_\Omega \z \cdot D^s w - \int_\Omega f^0(x, D^s u)
\end{align}
\begin{align}
&\leq \int_\Omega f(x, \nabla w) \, dx - \int_\Omega f(x, \nabla u) \, dx + \int_\Omega f^0(x, D^s w) - \int_\Omega f^0(x, D^s u) \\
&= \int_\Omega f(x,Dw)  -  \int_\Omega f(x,Du),
\end{align}
which concludes the proof.
\end{proof}

\begin{theorem}\label{thm:neumann}
Under the assumptions \ref{A1}-\ref{A2}, we have
$$\partial \mathcal{F_N} = \mathcal{A_N},$$
and $D(\mathcal{A_N})$ is dense in $L^2(\Omega)$.
\end{theorem}

\begin{proof}
{\bf  Step 1.} By Lemma \ref{lm:Neumann}, we have that the operator $\mathcal{A_N}$ is monotone and contained in $\partial \mathcal{F}_{\mathcal{N}}$. The operator $\partial \mathcal{F_N}$ is maximal monotone. Then, if $\mathcal{A_N}$ satisfies the range condition, by Minty Theorem we would also have that the operator $\mathcal{A_N}$ is maximal monotone, and consequently $\partial \mathcal{F_N}= \mathcal{A_N}$. In order to finish the proof, we need to show that $\mathcal{A_N}$ satisfies the range condition, i.e.
\begin{equation}\label{RCondN}
\hbox{Given} \ g \in  L^2(\Omega), \ \exists \, u \in D(\mathcal{A_N}) \ s.t. \ \  u + \mathcal{A_N}(u) \ni g.
\end{equation}
We can rewrite the above as
$$u + \mathcal{A_N}(u) \ni g \iff (u, g-u) \in \mathcal{A_N},$$
so we need to show that there exists a bounded vector field $\z \in X_2(\Omega)$  such that the following conditions hold:
$$ \z \in \partial_\xi f(x,\xi) \quad \mathcal{L}^N-\hbox{a.e. in } \Omega; $$
$$  -\mathrm{div}(\z) = g-u \quad \hbox{in} \ \Omega; $$
$$ \z \cdot D^s u = f^0(x,D^s u) \quad \hbox{as measures};$$
$$ [\z, \nu_\Omega] =0 \qquad  \mathcal{H}^{N-1}-\mbox{a.e.  on} \ \partial \Omega.$$
We are going to prove \eqref{RCondN}  by means of the Fenchel-Rockafellar Duality Theorem. We set
$$U = W^{1,1}(\Omega) \cap L^2(\Omega), \quad V = L^1(\partial\Omega,\mathcal{H}^{N-1}) \times L^1(\Omega;\mathbb{R}^N),$$
and the operator $A: U \rightarrow V$ is defined by the formula
$$ Au = (u|_{\partial\Omega}, \nabla u).$$
Clearly, $A$ is a linear and continuous operator. Moreover, the dual spaces to $U$ and $V$ are
$$
U^* = (W^{1,1}(\Omega) \cap L^2(\Omega))^*, \qquad V^* =  L^\infty(\partial\Omega,\mathcal{H}^{N-1}) \times L^\infty(\Omega;\mathbb{R}^N).
$$
We denote the points $p \in V$ in the following way: $p = (p_0, \overline{p})$, where $p_0 \in L^1(\partial\Omega,\mathcal{H}^{N-1})$ and $\overline{p} \in L^1(\Omega;\mathbb{R}^N)$. We will also use a similar notation for points $p^* \in V^*$. Then, we set $E: L^1(\partial\Omega, \mathcal{H}^{N-1}) \times L^1(\Omega;\mathbb{R}^N) \rightarrow \mathbb{R}$ by the formula
\begin{equation}\label{Ieq:definitionofEN}
E(p_0, \overline{p}) = E_0(p_0) + E_1(\overline{p}), \quad E_0(p_0) = 0, \quad E_1(\overline{p}) = \int_\Omega f(x,\overline{p}) \, dx.
\end{equation}
We also define $G:W^{1,1}(\Omega) \cap L^2(\Omega) \rightarrow \mathbb{R}$ as
$$G(u):= \frac12 \int_\Omega u^2 \, dx - \int_\Omega ug \, dx.$$

{\bf \flushleft Step 2.}  We now consider the convex conjugates $E^*$ and $G^*$. Notice that $G^*$ only enters the calculation via $A^* p^*$. First, observe that whenever $u^* \in U^*$ is such that $G^*(u^*) < \infty$, it holds that $u^* \in L^2(\Omega)$; to this end, we apply Lemma~\ref{ap.a.lem.dual} to the spaces $X = W^{1,1}(\Omega)$ and $Y = L^2(\Omega)$, with
$$\ell(u) := G(u) = \frac12 \int_{\Omega} u^2 \, dx - \int_{\Omega} ug \, dx$$
and the upper bound $\ell_0$ given by
$$ \ell_0(t) = \| g \|_{L^2(\Omega)} t + \frac{1}{2} t^2,$$
which yields the claim.

 Now, take any $p^* = (p^*_0,\overline{p}^*) \in L^\infty(\partial\Omega,\mathcal{H}^{N-1}) \times L^\infty(\Omega;\mathbb{R}^N)$ in the domain of $A^*$ with $G^*(A^*p^*) < \infty$. By the previous paragraph, it holds that $A^* p^* \in L^2(\Omega)$. Since the dual of the gradient is minus divergence, we get
\begin{equation}\label{eq:neumanndiv0}
A^* p^* =- \mathrm{div}(\overline{p}^*).
\end{equation}
In particular, $\overline{p}^* \in X_2(\Omega)$. Therefore,  if $G^*(A^*p^*) < \infty$, for any $u \in W^{1,1}(\Omega) \cap L^2(\Omega)$ we may apply the Green formula (Theorem \ref{Green}) and get
\begin{align}
\int_\Omega u \, (A^* p^*) \, dx &=  \langle u, A^* p^* \rangle_{U,U^*}  = \langle Au, p^* \rangle_{V,V^*}  = \int_{\partial\Omega} p^*_0 \, u \, d\mathcal{H}^{N-1} + \int_\Omega \overline{p}^* \cdot \nabla u \, dx \\
& = \int_{\partial\Omega} p^*_0 \, u \, d\mathcal{H}^{N-1} - \int_\Omega u \, \mathrm{div}(\overline{p}^*) \, dx + \int_{\partial\Omega} u \, [\overline{p}^*, \nu_\Omega] \, d\mathcal{H}^{N-1} \\
&= - \int_\Omega u \, \mathrm{div}(\overline{p}^*) \, dx + \int_{\partial\Omega} u \, (p_0^*  + [\overline{p}^*, \nu_\Omega]) \, d\mathcal{H}^{N-1}.
\end{align}
By \eqref{eq:neumanndiv0}, the integrals over $\Omega$ cancel out,  so
$$ \int_{\partial\Omega} u \, (p_0^* + [\overline{p}^*, \nu_\Omega]) \, d\mathcal{H}^{N-1} = 0$$
for all $u \in W^{1,1}(\Omega) \cap L^2(\Omega)$. Since the trace operator on $W^{1,1}(\Omega)$ is onto $L^1(\partial\Omega, \mathcal{H}^{N-1})$, by considering truncations we see that after a restriction to $W^{1,1}(\Omega) \cap L^\infty(\Omega)$ its image is $L^\infty(\partial\Omega,\mathcal{H}^{N-1})$. Therefore, after a restriction to $W^{1,1}(\Omega) \cap L^2(\Omega)$ its image is dense in $L^1(\partial\Omega,\mathcal{H}^{N-1})$, because it contains $L^\infty(\partial\Omega,\mathcal{H}^{N-1})$. In particular, we have
$$ \int_{\partial\Omega} w \, (p_0^*  + [\overline{p}^*, \nu_\Omega]) \, d\mathcal{H}^{N-1} = 0$$
for $w$ in a dense subset of $L^1(\partial\Omega,\mathcal{H}^{N-1})$. Hence,
\begin{equation}\label{eq:pandpoagree}
p_0^* = - [\overline{p}^*, \nu_\Omega] \qquad \mathcal{H}^{N-1}-\mbox{a.e. on } \partial\Omega.
\end{equation}
 We now turn to the functional $E^*$. Since the variables are separated, we have that $E^* = E_0^* + E_1^*$, and it is clear that the functional $E_0^*: L^\infty(\partial\Omega,\mathcal{H}^{N-1}) \rightarrow \mathbb{R} \cup \{ \infty \}$ is given by the formula
$$E_0^*(p_0^*) = \left\{ \begin{array}{ll} 0 \quad &\hbox{if} \ \ p_0^* = 0; \\[10pt] +\infty \quad & \hbox{if} \ \ p_0^* \not=  0.  \end{array}  \right. $$
The functional $E_1^*: L^\infty(\Omega; \mathbb{R}^N) \rightarrow [0,\infty]$ is given by the formula (see \cite[Proposition IV.1.2]{EkelandTemam})
$$E_1^*(\overline{p}^*) = \int_\Omega f^*(x,\overline{p}^*) \, dx. $$
Here, it is possible that $f^*$ takes the value $+\infty$ (in fact, for $1$-homogeneous functionals, it only takes the values $0$ and $+\infty$).

{\flushleft \bf Step 3.} Consider the energy functional $\mathcal{G_N} : L^2(\Omega) \rightarrow (-\infty, + \infty]$ defined by
\begin{equation}\label{11fuctmeN}
\mathcal{G_N}(v):= \left\{ \begin{array}{ll} \mathcal{F_N}(v) + G(v)  \quad &\hbox{if} \ v \in BV(\Omega) \cap L^2(\Omega); \\ \\ + \infty \quad &\hbox{if} \ v \in  L^2(\Omega) \setminus BV(\Omega).\end{array}\right.
\end{equation}
This functional is the extension of the functional  $E \circ A + G$, which is well-defined for functions in $W^{1,1}(\Omega) \cap L^2(\Omega)$, to the space $BV(\Omega) \cap L^2(\Omega)$. Since $\mathcal{G_N}$ is coercive, convex and lower semicontinuous, the primal minimisation problem
\begin{equation}\label{eq:primalproblemneumann}
\min_{v \in L^2(\Omega)} \mathcal{G}_{\mathcal{N}}(u) = \min_{v \in BV(\Omega) \cap L^2(\Omega)} \bigg\{ E(Av) + G(v) \bigg\}
\end{equation}
admits a solution $u$. Also, for $u_0 \equiv 0$ we have $E(Au_0) = 0 < \infty$, $G(u_0) = 0 < \infty$ and $E$ is continuous at $0$. Then, by the Fenchel-Rockafellar Duality Theorem, we have
\begin{equation}\label{DFR1-TVflowN}\inf \eqref{eq:primal} = \sup \eqref{eq:dual}
\end{equation}
and
\begin{equation}\label{DFR2-TVflowN} \hbox{the dual problem \eqref{eq:dual} admits at least one solution,}
\end{equation}
where the dual problem is
\begin{equation}\label{dual1N}
\sup_{p^* \in L^\infty(\partial\Omega, \mathcal{H}^{N-1}) \times L^\infty(\Omega; \mathbb{R}^N)} \bigg\{  - E_0^*(-p_0^*) - E_1^*(-\overline{p}^*) - G^*(A^* p^*) \bigg\}.
\end{equation}
Keeping in mind the above calculations, we set $\mathcal{Z}$ to be the subset of $V^*$ such that the dual problem does not immediately return $-\infty$, namely
$$
\mathcal{Z} = \bigg\{ p^* \in L^\infty(\partial\Omega, \mathcal{H}^{N-1}) \times L^\infty(\Omega; \mathbb{R}^N): \quad \mathrm{div}(\overline{p}^*) \in L^2(\Omega), \quad p_0^* =0 \bigg\}.
$$
Hence, we may rewrite the dual problem as
\begin{equation}\label{dual3N}
\sup_{p^* \in \mathcal{Z}} \bigg\{ - E_1^*(-\overline{p}^*) - G^*(A^* p^*)\bigg\}.
\end{equation}
Now, let us take a sequence $u_n \in W^{1,1}(\Omega)$ which has the same trace as $u$ and converges strictly to $u$,  a minimiser of $\mathcal{G_N}$, and also $u_n \to u$ in $L^2(\Omega)$; then, it is a minimising sequence in \eqref{eq:primal}.  Since by \eqref{DFR1-TVflowN} and \eqref{DFR2-TVflowN} there is no duality gap and there exists a solution to the dual problem, we can now apply to this sequence the $\varepsilon$-subdifferentiability property given in \eqref{eq:epsilonsubdiff11N} and \eqref{eq:epsilonsubdiff21N}.  Let $p^*$ be a solution to the dual problem. By equation \eqref{eq:epsilonsubdiff21N}, for every $w \in L^2(\Omega)$ we have
$$G(w) - G(u_n) \geq \langle (w -u_n), A^* p^* \rangle_{ U,U^*}  -\varepsilon_n.$$
Hence, since $A^*p^* \in L^2(\Omega)$,
$$G(w) - G(u) \geq \langle (w -u), A^* p^* \rangle_{ U,U^*}  = \langle (w -u), A^* p^* \rangle_{L^2(\Omega)},$$
and consequently,
$$ A^* p^* \in  \partial_{L^2(\Omega)} G(u) = \{ u - g \}. $$
Therefore, by \eqref{eq:neumanndiv0} we get
\begin{equation}\label{div1N}
 -\mathrm{div}(\overline{p}^*) = u -g.
\end{equation}
Also, by the definition of $\mathcal{Z}$ and keeping in mind that $p_0^* = - [\overline{p}^*, \nu_\Omega]$, we get
\begin{equation}\label{front1N}
[-\overline{p}^*, \nu_\Omega] =0 \quad \mathcal{H}^{N-1}-\mbox{a.e. on} \ \partial \Omega.
\end{equation}
Therefore, since the boundary term disappears, equation \eqref{eq:epsilonsubdiff11N} gives
\begin{equation}\label{eq:goodN}
0 \leq \int_\Omega f(x,\nabla u_n)  \, dx + \int_\Omega f^*(x,\overline{p}^*) \, dx - \langle \nabla u_n, -\overline{p}^* \rangle_{L^1(\Omega;\mathbb{R}^N),L^\infty(\Omega;\mathbb{R}^N)}  \leq \varepsilon_n,
\end{equation}
which we rewrite as
\begin{equation}\label{eq:neumannbeforelimit}
0 \leq \int_\Omega f(x,\nabla u_n)  \, dx + \int_\Omega f^*(x,\overline{p}^*)  \, dx \leq \int_\Omega \nabla u_n \cdot (-\overline{p}^*) \, dx + \varepsilon_n.
\end{equation}
Keeping in mind that $-\mathrm{div}(\overline{p}^*) = u-g$ and again using the fact that the trace of $u_n$ is fixed and equal to the trace of $u$, by Green's formula we get
$$ \int_\Omega  \nabla u_n \cdot \overline{p}^* \, dx  = - \int_\Omega u_n \,  \mathrm{div}(\overline{p}^*)\, dx   = \int_\Omega u_n \, (u -g)  \, dx. $$
Then, applying again Green's formula, we have
$$\lim_{n \to \infty}  \int_\Omega \nabla u_n \cdot \overline{p}^* \, dx = \int_\Omega u \, (u-g) \, dx = -  \int_\Omega u \,  \mathrm{div}(\overline{p}^*)  \, dx  = \int_\Omega (\overline{p}^*, Du). $$
But then, passing to the limit in equation \eqref{eq:neumannbeforelimit}, by Reschetnyak theorem (Theorem \ref{thm:reshetnyak}) we get that
\begin{equation}
\int_\Omega f(\cdot, \nabla u) \, dx + \int_{\overline{\Omega}} f^0\left(\cdot, \frac{d D^s u}{d|D^s u|}\right) d|D^s u| + \int_\Omega f^*(x,\overline{p}^*) \, dx = \int_\Omega -\overline{p}^* \cdot \nabla u \, dx + \int_\Omega (-\overline{p}^*, Du)^s.
\end{equation}
From this, the required results on the absolutely continuous part and the singular part follow. Indeed, by the definition of the dual function
\begin{equation}
\int_\Omega f(\cdot, \nabla u) \, dx + \int_\Omega f^*(x,\overline{p}^*) \, dx \geq \int_\Omega -\overline{p}^* \cdot \nabla u \, dx,
\end{equation}
and by Proposition \ref{prop:resultfrombs} we get that \begin{equation}
\int_{\overline{\Omega}} f^0 \left(\cdot, \frac{d D^s u}{d|D^s u|}\right) d|D^s u| \geq \int_\Omega (-\overline{p}^*, Du)^s,
\end{equation}
so both of these inequalities have to be equalities.  But, the first one means exactly that $-\overline{p}^* \in \partial_\xi f(x,\nabla u)$, so \eqref{e1Def} holds for the choice $\mathbf{z} = -\overline{p}^*$, and the second one means that $-\overline{p}^* \cdot D^s u = f^0(x,D^s u)$, hence \eqref{e3Def} holds. We already have the divergence constraint \eqref{e2Def} by \eqref{div1N} and the boundary constraint \eqref{e4Def} is incorporated in the definition of the dual problem.

Finally, by \cite[Proposition 2.11]{Brezis}, we have
$$ D(\partial \mathcal{F_N}) \subset  D(\mathcal{F_N}) =  BV(\Omega) \cap L^2(\Omega) \subset \overline{D(\mathcal{F_N})}^{L^2(\Omega)} \subset \overline{D(\partial \mathcal{F_N})}^{L^2(\Omega)},$$
from which follows the density of the domain.
\end{proof}

Our concept of solution of the Neumann problem \eqref{NeumannProb} is the following:
\begin{definition}\label{def:neumann1p}
{\rm  Given $u_0 \in L^2(\Omega)$, we say that $u$ is a {\it weak solution} of the  Neumann problem \eqref{NeumannProb} in $[0,T]$, if $u \in  C([0,T]; L^2(\Omega)) \cap W_{loc}^{1,2}(0, T; L^2(\Omega))$,   $u(0, \cdot) =u_0$, and for almost all $t \in (0,T)$
\begin{equation}\label{def:neumannflow}
u_t(t, \cdot) + \mathcal{A_N}u(t, \cdot) \ni 0.
\end{equation}
In other words, if $u(t) \in BV(\Omega)$ and there exist vector fields $\z(t) \in X_2(\Omega)$ such that the following conditions hold:
\begin{equation}\label{e1DefP}\z(t) \in \partial_\xi f(x,\nabla u(t)) \quad \mathcal{L}^N-\hbox{a.e. in } \Omega;\end{equation}
\begin{equation}
u_t(t) = \mathrm{div}(\z(t))\quad  \hbox{in} \ \mathcal{D}^\prime(\Omega);\end{equation}
\begin{equation}
\z(t) \cdot D^s u(t) = f^0(x,D^s u(t)) = f^0(x,\overrightarrow{D^s u(t)}) \vert
D^s u(t) \vert \quad \hbox{as measures};\end{equation}
\begin{equation}\label{e4DefP}[\z(t), \nu_\Omega] =0 \qquad  \mathcal{H}^{N-1}-\mbox{a.e. on } \partial \Omega.
\end{equation}

}
\end{definition}

Then, using the classical theory of maximal monotone operators (see for instance \cite{Brezis}), as a consequence of Theorem \ref{thm:neumann} we have the following existence and uniqueness theorem.

\begin{theorem}
Under the assumptions \ref{A1}-\ref{A2}, for any $u_0 \in L^2(\Omega)$ and all $T > 0$ there exists a unique weak solution of the Neumann problem \eqref{NeumannProb} in $[0,T]$.
\end{theorem}

Let us now briefly discuss how the general Definition \ref{def:neumann1p} looks like when specified to the most important special cases.

\begin{example}\label{ex:tvflow}
Take $f(x,\xi) = |\xi|$, which is the Lagrangian corresponding to the total variation flow. Then, the condition that the vector field $\z$ lies in the respective subdifferential a.e. implies that $\| \z \|_\infty \leq 1$. Therefore, given $u_0 \in L^2(\Omega)$, a function $u$ is a {\it weak solution} to the Neumann problem for the total variation flow in $[0,T]$, if $u \in  C([0,T]; L^2(\Omega)) \cap W_{loc}^{1,2}(0, T; L^2(\Omega))$,  $u(0, \cdot) =u_0$, and for almost all $t \in (0,T)$ we have $u(t) \in BV(\Omega)$ and there exist vector fields $\z(t) \in X_2(\Omega)$ with $\| \z(t) \|_\infty \leq 1$ such that the following conditions hold:
\begin{equation}
u_t(t) = \mathrm{div}(\z(t))\quad  \hbox{in} \ \mathcal{D}^\prime(\Omega);
\end{equation}
\begin{equation}
(\z(t), Du(t)) = |Du(t)| \quad \hbox{as measures};
\end{equation}
\begin{equation}
[\z(t), \nu_\Omega] =0 \qquad  \mathcal{H}^{N-1}-\mbox{a.e. on } \partial \Omega.
\end{equation}
For the anisotropic total variation flow, i.e. when $f(x,\xi)$ is $1$-homogeneous with respect to the second variable and comparable to the Euclidean norm, the condition that $\| \z(t) \|_\infty \leq 1$ is replaced by an estimate on the polar norm of $\z(t)$ and the total variation is computed with respect to an anisotropy.
\end{example}

\begin{example}\label{ex:c1functional}
Suppose that $f$ is a convex differentiable function of $\xi$ with continuous gradient for each fixed $x \in \Omega$. Then, since the subdifferential of $f$ with respect to $\xi$ contains only a single element, we denote
\begin{equation}
\a(x,\xi) := \nabla_\xi f(x, \xi)
\end{equation}
and the definition of solutions reduces to the following one: given $u_0 \in L^2(\Omega)$, a function $u$ is a {\it weak solution} to the Neumann problem for the gradient flow of $F$ in $[0,T]$, if $u \in  C([0,T]; L^2(\Omega)) \cap W_{loc}^{1,2}(0, T; L^2(\Omega))$,  $u(0, \cdot) =u_0$, and for almost all $t \in (0,T)$ we have $u(t) \in BV(\Omega)$, $\a(x,\nabla u(t)) \in X_2(\Omega)$ and the following conditions hold:
\begin{equation}
u_t(t) = \mathrm{div}(\a(x,\nabla u(t)))\quad  \hbox{in} \ \mathcal{D}^\prime(\Omega);
\end{equation}
\begin{equation}
\a(x,\nabla u(t)) \cdot D^s u(t) = f^0(x, D^s u(t));
\end{equation}
\begin{equation}
[\a(x,\nabla u(t)), \nu_\Omega] = 0 \qquad  \mathcal{H}^{N-1}-\mbox{a.e. on } \partial \Omega.
\end{equation}
\end{example}

\begin{remark}
Working similarly as in the previous results in this Section, we can get a characterisation of solutions to the Cauchy problem in $\mathbb{R}^N$. The energy functional $\mathcal{F}_{\mathcal{C}} : L^2(\mathbb{R}^N) \rightarrow [0, + \infty]$ takes the form
\begin{equation}
\mathcal{F}_{\mathcal{C}}(u):= \left\{ \begin{array}{ll} \displaystyle\int_{\mathbb{R}^N} f(x,Du)  \quad &\hbox{if} \ u \in BV(\mathbb{R}^N) \cap L^2(\mathbb{R}^N); \\ \\ + \infty \quad &\hbox{if} \ u \in  L^2(\mathbb{R}^N) \setminus BV(\mathbb{R}^N).
\end{array}\right.
\end{equation}
 Assuming that $f$ satisfies a growth condition similar to \ref{A1}, i.e. $|f(x,\xi)| \leq \psi(x) + M |\xi|$ with $\psi \in L^1(\mathbb{R}^N)$, there exists a unique strong solution of the abstract Cauchy problem
\begin{equation}\label{eq:cauchyproblem}
\left\{ \begin{array}{ll} u'(t) + \partial \mathcal{F}_{\mathcal{C}}(u(t)) \ni 0, \quad t \in [0,T] \\[5pt] u(0) = u_0. \end{array}\right.
\end{equation}
 If we additionally assume that condition \ref{A2} holds, one can prove that the characterisation of the subdifferential of $\mathcal{F}_{\mathcal{C}}$ takes the following form: $(u,v) \in \partial \mathcal{F_C}$ if and only if $u, v \in L^2(\mathbb{R}^N)$, $u \in BV(\mathbb{R}^N)$ and  there exists a vector field $\z \in X_2(\mathbb{R}^N)$ such that the following conditions hold:
\begin{equation}\label{e1DefCauchy}\z \in \partial_\xi f(x,\nabla u) \quad \mathcal{L}^N-\hbox{a.e. in } \mathbb{R}^N;\end{equation}
\begin{equation}\label{e2DefCauchy} -\mathrm{div}(\z) = v \quad \hbox{in} \ \mathbb{R}^N;\end{equation}
\begin{equation}\label{e3DefCauchy}\z \cdot D^s u = f^0(x,D^s u) \quad \hbox{as measures};\end{equation}
Given $u_0 \in L^2(\mathbb{R}^N)$, we say that $u$ is a {\it weak solution} of the Cauchy problem \eqref{eq:cauchyproblem} in $[0,T]$, if $u \in  C([0,T]; L^2(\mathbb{R}^N)) \cap W_{loc}^{1,2}(0, T; L^2(\mathbb{R}^N))$,   $u(0, \cdot) =u_0$, and for almost all $t \in (0,T)$ we have $u(t) \in BV(\mathbb{R}^N)$ and there exist vector fields $\z(t) \in X_2(\mathbb{R}^N)$ such that the following conditions hold:
\begin{equation}
\z(t) \in \partial_\xi f(x,\nabla u(t)) \quad \mathcal{L}^N-\hbox{a.e. in } \mathbb{R}^N;\end{equation}
\begin{equation}\label{e2DefP}  u_t(t) = \mathrm{div}(\z(t))\quad  \hbox{in} \ \mathcal{D}^\prime(\mathbb{R}^N);\end{equation}
\begin{equation}\label{e3DefP} \z(t) \cdot D^s u(t) = f^0(x,D^s u(t)) = f^0(x,\overrightarrow{D^s u(t)}) \vert
D^s u(t) \vert \quad \hbox{as measures}.
\end{equation}
Then, for all $u_0 \in L^2(\mathbb{R}^N)$ and all $T > 0$ there exists a unique weak solution to the Cauchy problem \eqref{eq:cauchyproblem} in $[0,T]$.
\end{remark}

\section{The Dirichlet Problem}\label{sec:dirichlet}

Recall that throughout the whole paper we assume that $\Omega$ is an open bounded subset of $\R^N$ with Lipschitz boundary, where $N \geq 2$. Moreover, the Lagrangian  $f \in C(\overline{\Omega} \times \mathbb{R}^N)$ is convex in the variable $\xi$ and satisfies conditions \ref{A1} and \ref{A2}. In this section, we study the Dirichlet problem
\begin{equation}\label{DirichletProb}
\left\{ \begin{array}{lll} u_t (t,x) = {\rm div} (\partial_\xi f(x, Du(t,x)))   \quad &\hbox{in} \ \ (0, T) \times \Omega; \\[5pt] u(t) = h \quad &\hbox{on} \ \ (0, T) \times  \partial\Omega; \\[5pt] u(0,x) = u_0(x) \quad & \hbox{in} \ \  \Omega, \end{array} \right.
\end{equation}
where $u_0 \in L^2(\Omega)$ and $h \in L^1(\partial\Omega, \mathcal{H}^{N-1})$. Similarly to the previous section, when $f \in C^1(\overline{\Omega} \times \R^N)$, the subdifferential is single-valued and when we set $\a(x,\xi) = \partial_\xi f(x, \xi)$, the problem reduces to the classical case
\begin{equation}
\left\{ \begin{array}{lll} u_t (t,x) = {\rm div}\,  \a(x, Du(t,x))   \quad &\hbox{in} \ \ (0, T) \times \Omega; \\[5pt] u(t) = h \quad &\hbox{on} \ \ (0, T) \times  \partial\Omega; \\[5pt] u(0,x) = u_0(x) \quad & \hbox{in} \ \  \Omega \end{array} \right.
\end{equation}
treated in \cite{ACMIbero} (see also \cite{ACMBook}), under a bit more restrictive assumptions than the ones used in this paper, and for the choice $f(x,\xi) = |\xi|$ we recover the total variation flow
\begin{equation}
\left\{ \begin{array}{lll} u_t (t,x) = {\rm div} \bigg(\frac{Du(t,x)}{|Du(t,x)|} \bigg)  \quad &\hbox{in} \ \ (0, T) \times \Omega; \\[5pt] u(t) = h \quad &\hbox{on} \ \ (0, T) \times  \partial\Omega; \\[5pt] u(0,x) = u_0(x) \quad & \hbox{in} \ \  \Omega \end{array} \right.
\end{equation}
that was studied in \cite{ABCM2}. As previously, our definition will be designed similarly to the one for the total variation flow, using a vector field $\z \in X_2(\Omega)$ to replace the possibly multivalued gradient $\partial_\xi f(x, Du(t,x))$.

Consider the energy functional $\mathcal{F}_{h} : L^2(\Omega) \rightarrow [0, + \infty]$ associated with problem \eqref{DirichletProb} and defined by
\begin{equation}
\mathcal{F}_{h}(u):= \left\{ \begin{array}{ll} \displaystyle\int_\Omega f(x,Du) + \int_{\partial\Omega} f^0(x,\nu_\Omega) \, |h - u| \, d\mathcal{H}^{N-1}  \quad &\hbox{if} \ u \in BV(\Omega) \cap L^2(\Omega); \\ \\ + \infty \quad &\hbox{if} \ u \in  L^2(\Omega) \setminus BV(\Omega).
\end{array}\right.
\end{equation}
Note that by \eqref{cpndI} the integral on the boundary can be written in the form
$$f^0(x, (h -u) \nu_\Omega) =f^0(x,\nu_\Omega) \, |h - u|.$$
Recall that by the results in \cite{Anzellotti2} (see Section \ref{sec:functionofmeasure}), under the assumption that $\tilde{g}(x, \xi, t)$ is continuous on $\overline{\Omega} \times \R^N \times [0, + \infty)$ and convex in $(\xi,t)$ for each fixed $x$, the functional $\mathcal{F}_{h}$ is lower semicontinuous respect to the $L^1$-convergence. Moreover, since $\mathcal{F}_{h}$ is convex, by the theory of maximal monotone operators (see \cite{Brezis}) there is a unique strong solution of the abstract Cauchy problem
\begin{equation}\label{CPDirichlet}
\left\{ \begin{array}{ll} u'(t) + \partial \mathcal{F}_{h}(u(t)) \ni 0, \quad t \in [0,T] \\[5pt] u(0) = u_0. \end{array}\right.
\end{equation}
To characterise the subdifferential of $\mathcal{F}_{h}$, we define the following operator.

\begin{definition}{\rm We say that $(u,v) \in \mathcal{A}_h$ if and only if $u, v \in L^2(\Omega)$, $u \in BV(\Omega)$ and  there exists a vector field $\z \in  X_2(\Omega)$  such that the following conditions hold:
\begin{equation}\label{e1DefD} \z \in \partial_\xi f(x,\nabla u) \quad \mathcal{L}^N-\hbox{a.e. in } \Omega;
\end{equation}
\begin{equation}\label{e2DefD}  -\mathrm{div}(\z) = v \quad \hbox{in} \ \Omega;
\end{equation}
\begin{equation}\label{e3DefD} \z \cdot D^s u = f^0(x,D^s u) \quad \hbox{as measures};
\end{equation}
\begin{equation}\label{e4DefD} [\z, \nu_\Omega] \in \mbox{sign}(h-u) f^0(x,\nu_\Omega) \qquad  \mathcal{H}^{N-1}-\mbox{a.e.  on} \ \partial \Omega.
\end{equation}
}
\end{definition}

\begin{lemma}\label{lm:Dirichlet} Under the assumptions \ref{A1}-\ref{A2}, we have
\begin{equation}\label{e3.lm:Dirichlet}
\mathcal{A}_h \subset \partial \mathcal{F}_h.
\end{equation}
\end{lemma}

\begin{proof}
Let $(u,v) \in \mathcal{A}_h$ and $\z \in X_2(\Omega)$ satisfying \eqref{e1DefD}-\eqref{e4DefD}. Given $w \in L^2(\Omega) \cap BV(\Omega)$, by \eqref{e1DefD} we have
\begin{equation}\label{e1.lm:Dirichlet}
f(x,\nabla w(x)) - f(x,\nabla u(x)) \geq \z(x) \cdot (\nabla w(x) - \nabla u(x)), \quad \quad  \mathcal{L}^N-\hbox{a.e. in } \Omega.
\end{equation}
Then, since $\z \in X_2(\Omega)$ satisfies \eqref{e1DefD}-\eqref{e4DefD}, applying Green formula \eqref{Green} and having in mind properties \eqref{good1}, \eqref{good2} and \eqref{e1.lm:Dirichlet}, we have
\begin{align}
\int_\Omega (w - u) \, v \, dx &= - \int_\Omega \mathrm{div}(\z) (w - u) \, dx \\
&= \int_\Omega (\z, Dw) - \int_{\partial\Omega} w \, [\z, \nu_\Omega] \, d\mathcal{H}^{N-1} - \int_\Omega (\z, Du) + \int_{\partial\Omega} u \, [\z, \nu_\Omega] \, d\mathcal{H}^{N-1} \\
&= \int_\Omega \z \cdot \nabla w \, dx + \int_\Omega \z \cdot D^s w - \int_\Omega \z \cdot \nabla u \, dx - \int_\Omega \z \cdot D^s u \\
&\qquad\qquad\qquad  - \int_{\partial\Omega} (w - h) \, [\z, \nu_\Omega] \, d\mathcal{H}^{N-1} + \int_{\partial\Omega} (u - h) \, [\z, \nu_\Omega] \,  d\mathcal{H}^{N-1}
\end{align}
\begin{align}
&\leq \int_\Omega f(x, \nabla w) \, dx - \int_\Omega f(x, \nabla u) \, dx + \int_\Omega \z \cdot D^s w - \int_\Omega f^0(x, D^s u) \\
&\qquad\qquad\qquad - \int_{\partial\Omega} (w - h) \, [\z, \nu_\Omega] \, d\mathcal{H}^{N-1} - \int_{\partial\Omega} f^0(x,\nu_\Omega) \, |h - u| \, d\mathcal{H}^{N-1} \\
&\leq \int_\Omega f(x, \nabla w) \, dx + \int_\Omega f^0(x, D^s w) - \int_\Omega f(x, \nabla u) \, dx -\int_\Omega f^0(x, D^s u) \\
&\qquad\qquad\qquad + \int_{\partial\Omega} f^0(x,\nu_\Omega) |h-w| \, d\mathcal{H}^{N-1} - \int_{\partial\Omega} f^0(x,\nu_\Omega) \, |h - u| \, d\mathcal{H}^{N-1} \\
&= \int_\Omega f(x,Dw) + \int_{\partial\Omega} f^0(x,\nu_\Omega) |h-w| \, d\mathcal{H}^{N-1} \\
&\qquad\qquad\qquad -  \int_\Omega f(x,Du) - \int_{\partial\Omega} f^0(x,\nu_\Omega) \, |h - u| \, d\mathcal{H}^{N-1} \\
&= \mathcal{F}_{h}(w) - \mathcal{F}_{h}(u),
\end{align}
which concludes the proof.
\end{proof}

\begin{theorem}\label{thm:Dirichlet}
Suppose that $f^0(x,\nu_\Omega) \geq 0$ $\mathcal{H}^{N-1}$-a.e. on $\partial\Omega$. Then, under the assumptions \ref{A1}-\ref{A2} we have
$$\partial \mathcal{F}_h = \mathcal{A}_h,$$
and $D(\mathcal{A}_h)$ is dense in $L^2(\Omega)$.
\end{theorem}

\begin{proof}
{\bf  Step 1.} Like in the proof of Theorem \ref{thm:neumann}, we only need to show that $\mathcal{A}_h$ satisfies the range condition, i.e.
\begin{equation}\label{RCondD}
\hbox{Given} \ g \in  L^2(\Omega), \ \exists \, u \in D(\mathcal{A}_h) \ s.t. \ \  u + \mathcal{A}_h(u) \ni g.
\end{equation}
 We can rewrite the above as
$$u + \mathcal{A}_h(u) \ni g \iff (u, g-u) \in \mathcal{A}_h,$$
so we need to show that there exists a bounded vector field $\z \in X_2(\Omega)$  such that the following conditions hold:
$$ \z \in \partial_\xi f(x,\xi) \quad \mathcal{L}^N-\hbox{a.e. in } \Omega; $$
$$  -\mathrm{div}(\z) = g-u \quad \hbox{in} \ \Omega; $$
$$ \z \cdot D^s u = f^0(x,D^s u) \quad \hbox{as measures};$$
$$ [\z, \nu_\Omega] \in \mbox{sign}(h-u) f^0(x,\nu_\Omega) \qquad  \mathcal{H}^{N-1}-\mbox{a.e. on} \ \partial \Omega.$$
As in the proof of Theorem \ref{thm:neumann}, we will prove \eqref{RCondD} by means of the Fenchel-Rockafellar duality theorem. We will employ the same function spaces and operator $A$ as in that proof. We will also use the same functionals $E_1$ and $G$, the only difference being that the functional $E_0$ now incorporates the Dirichlet boundary condition. For convenience, we shortly present again the whole framework.

We set $U = W^{1,1}(\Omega) \cap L^2(\Omega)$, $V = L^1(\partial\Omega,\mathcal{H}^{N-1}) \times L^1(\Omega;\mathbb{R}^N)$, and the operator $A: U \rightarrow V$ is defined by the formula
$$ Au = (u|_{\partial\Omega}, \nabla u).$$
Note that $A$ is a linear and continuous operator and the dual spaces to $U$ and $V$ are
$$ U^* = (W^{1,1}(\Omega) \cap L^2(\Omega))^*, \qquad V^* =  L^\infty(\partial\Omega,\mathcal{H}^{N-1}) \times L^\infty(\Omega;\mathbb{R}^N).$$
We denote the points $p \in V$ in the following way: $p = (p_0, \overline{p})$, where $p_0 \in L^1(\partial\Omega,\mathcal{H}^{N-1})$ and $\overline{p} \in L^1(\Omega;\mathbb{R}^N)$. We will also use a similar notation for points $p^* \in V^*$. Then, we set $E: L^1(\partial\Omega, \mathcal{H}^{N-1}) \times L^1(\Omega;\mathbb{R}^N) \rightarrow \mathbb{R}$ by the formula
\begin{equation}\label{Ieq:definitionofEND}
E(p_0, \overline{p}) = E_0(p_0) + E_1(\overline{p}), \quad E_0(p_0) = \int_{\partial \Omega} f^0(x, \nu_\Omega) \vert p_0 -  h \vert \,d\mathcal{H}^{N-1}, \quad E_1(\overline{p}) = \int_\Omega f(x,\overline{p}) \, dx.
\end{equation}
We also set $G:W^{1,1}(\Omega) \cap L^2(\Omega) \rightarrow \mathbb{R}$ by
$$G(u):= \frac12 \int_\Omega u^2 \, dx - \int_\Omega ug \, dx.$$

{\flushleft \bf Step 2.}  Since $G$ is defined by the same formula as in the proof of Theorem \ref{thm:neumann}, we again have that whenever $u^* \in U^*$ is such that $G^*(u^*) < \infty$, it holds that $u^* \in L^2(\Omega)$. Moreover, since the operator $A$ is defined by the same formula as in the proof of Theorem \ref{thm:neumann},  for any $p^* \in D(A^*)$ with $G^*(A^* p^*) < \infty$ conditions \eqref{eq:neumanndiv0} and \eqref{eq:pandpoagree} remain true, so
\begin{equation}\label{eq:dirichletdiv0}
A^* p^* =- \mathrm{div}(\overline{p}^*)
\end{equation}
and
\begin{equation}\label{eq:pandpoagreeN}
p_0^* = - [\overline{p}^*, \nu_\Omega] \qquad \mathcal{H}^{N-1}-\mbox{a.e. on } \partial\Omega.
\end{equation}
Now,  observe that $E^* = E_0^* + E_1^*$, and compute the convex conjugate of the functional~$E_0$. The functional $E_0^*: L^\infty(\partial\Omega, \mathcal{H}^{N-1}) \rightarrow \mathbb{R} \cup \{ \infty \}$ is given by the formula
\begin{equation}\label{eq:formulafore0star}
E_0^*(p_0^*) = \left\{ \begin{array}{ll} \displaystyle\int_{\partial\Omega} h \, p_0^* \, d\mathcal{H}^{N-1} \quad &\hbox{if} \ \  |p_0^*| \leq f^0(x,\nu_\Omega) \quad \mathcal{H}^{N-1}-\mbox{a.e. on } \partial\Omega; \\[10pt] +\infty \quad &\hbox{otherwise.}   \end{array}  \right.
\end{equation}
In fact, we only need to observe that
$$E_0^*(p_0^*) = \sup_{p_0 \in L^1(\partial\Omega, \mathcal{H}^{N-1})} \left\{ \int_{\partial \Omega} p_0 \, p_0^* \, d\mathcal{H}^{N-1} - \int_{\partial\Omega} f^0(x,\nu_\Omega) |p_0 - h| \, d\mathcal{H}^{N-1} \right\}$$
$$= \sup_{p_0 \in L^1(\partial\Omega,\mathcal{H}^{N-1})} \left\{ \int_{\partial \Omega} (p_0 - h) \, p_0^* \, d\mathcal{H}^{N-1} + \int_{\partial\Omega} h \, p_0^* \, d\mathcal{H}^{N-1}  - \int_{\partial\Omega} f^0(x, \nu_\Omega) |p_0 - h| \, d\mathcal{H}^{N-1} \right\},$$
and then formula \eqref{eq:formulafore0star} follows due to our assumption that $f^0(x,\nu_\Omega) \geq 0$.

The functional  $E_1^*: L^\infty(\Omega; \mathbb{R}^N) \rightarrow [0,\infty]$  was already computed in the proof of Theorem \ref{thm:neumann} and is given by the formula
$$E_1^*(\overline{p}^*) = \int_\Omega f^*(x,\overline{p}^*) \, dx. $$

{\flushleft \bf Step 3.} Now, we will apply the Fenchel-Rockafellar duality theorem to the primal problem of the form
\begin{equation}\label{primal1}
\inf_{u \in U} \bigg\{ E(Au) + G(u) \bigg\}
\end{equation}
with $E$ and $G$ defined as above. For $u_0 \equiv 0$ we have $E(Au_0) = \int_{\partial\Omega} |h| \, d\mathcal{H}^{N-1} < \infty$, $G(u_0) = 0 < \infty$ and $E$ is continuous at $0$. Then, by the Fenchel-Rockafellar duality theorem we have
\begin{equation}\label{DFR1-TVflow}\inf \eqref{eq:primal} = \sup \eqref{eq:dual}
\end{equation}
and
\begin{equation}\label{DFR2-TVflow} \hbox{the dual problem \eqref{eq:dual} admits at least one solution,}
\end{equation}
where the dual problem is given by
\begin{equation}\label{dual1}
\sup_{p^* \in L^\infty(\partial\Omega, \mathcal{H}^{N-1}) \times L^\infty(\Omega; \mathbb{R}^N)} \bigg\{ - E_0^*(-p_0^*) - E_1^*(-\overline{p}^*) - G^*(A^* p^*) \bigg\}.
\end{equation}
Keeping in mind the above calculations, we set $\mathcal{Z}$ to be the subset of $V^*$ such that the dual problem does not immediately return $-\infty$, namely
\begin{align}
\mathcal{Z} = \bigg\{ p^* \in L^\infty(\partial\Omega, \mathcal{H}^{N-1}) \times &X_2(\Omega): \quad f^*(\cdot,\overline{p}^*) < \infty \quad \mathcal{L}^N-\mbox{a.e in } \Omega; \\
& |p_0^*| \leq f^0(x,\nu_\Omega) \quad \mathcal{H}^{N-1}-\mbox{a.e. on } \partial\Omega; \quad p_0^* = -[\overline{p}^*, \nu_\Omega] \bigg\}.
\end{align}
Hence, we may rewrite the dual problem as
\begin{equation}\label{dual2}
\sup_{p^* \in \mathcal{Z}} \bigg\{ - E_0^*(-p_0^*)- E_1^*(-\overline{p}^*) - G^*(A^* p^*)\bigg\},
\end{equation}
so finally the dual problem takes the form
\begin{equation}\label{dual3}
\sup_{p^* \in \mathcal{Z}} \bigg\{ \int_{\partial\Omega} h \, p_0^* \, d\mathcal{H}^{N-1} + \int_\Omega f^*(x,\overline{p}^*) \, dx - G^*(A^* p^*)\bigg\}.
\end{equation}
In  fact, in light of the constraint $p_0^* = - [\overline{p}^*, \nu_\Omega]$, we may simplify the dual problem a bit: observe that the constraint that $|p_0^*| \leq f^0(x,\nu_\Omega)$ $\mathcal{H}^{N-1}$-a.e. on $\partial\Omega$ in the dual problem follows from the other two conditions due to property \eqref{good2}.

Now, consider the energy functional $\mathcal{G}_f : L^2(\Omega) \rightarrow (-\infty, + \infty]$ defined by
\begin{equation}\label{11fuctme}
\mathcal{G}_f (v):= \left\{ \begin{array}{ll} \mathcal{F}_h(v) + G(v)   \quad &\hbox{if} \ v \in BV(\Omega) \cap L^2(\Omega); \\ \\ + \infty \quad &\hbox{if} \ v \in  L^2(\Omega) \setminus BV(\Omega).\end{array}\right.
\end{equation}
As in the proof of Theorem \ref{thm:neumann}, this functional is the extension of the functional $E \circ A + G$ (well-defined for functions in $W^{1,1}(\Omega) \cap L^2(\Omega)$) to the space $BV(\Omega) \cap L^2(\Omega)$. Since $\mathcal{G}_h$ is coercive, convex and lower semicontinuous, the minimisation problem
$$\min_{v \in L^2(\Omega)} \mathcal{G}_h(v)$$ admits a solution $u$ and we have
\begin{equation}\label{minimun1}
\min_{v \in L^2(\Omega)} \mathcal{G}_h(v) = \min_{v \in U} \bigg\{ E(Av) + G(v) \bigg\}.
\end{equation}
Let us take a sequence $u_n \in W^{1,1}(\Omega)$ which has the same trace as $u$ and converges strictly to $u$,  a minimiser of $\mathcal{G}_h$, and also $u_n \to u$ in $L^2(\Omega)$; then, it is a minimising sequence in~\eqref{eq:primal}.  Since we have  \eqref{DFR1-TVflow} and \eqref{DFR2-TVflow}, we may use the $\varepsilon-$subdifferentiability property of minimising sequences given in \eqref{eq:epsilonsubdiff11N} and \eqref{eq:epsilonsubdiff21N}. Let $p^*$ be a solution of the dual problem. By equation \eqref{eq:epsilonsubdiff21N}, for every $w \in L^2(\Omega)$, we have
$$G(w) - G(u_n) \geq \langle (w -u_n), A^* p^* \rangle_{U,U^*}  -\varepsilon_n.$$
Since $G^*(A^* p^*) < \infty$, we have that $A^* p^* \in L^2(\Omega)$. Hence,
$$G(w) - G(u) \geq \langle (w -u), A^* p^* \rangle_{ U,U^*}  = \langle (w -u), A^* p^* \rangle_{L^2(\Omega)} ,$$
and consequently,
$$ A^* p^* \in  \partial_{L^2(\Omega)} G(u) = \{ u - g \}. $$
Therefore, by \eqref{eq:dirichletdiv0} we get
\begin{equation}\label{div1}
-\mathrm{div}(\overline{p}^*) = u -g.
 \end{equation}
On the other hand, equation \eqref{eq:epsilonsubdiff11N} gives
\begin{align}
0 \leq \int_{\partial\Omega} \bigg(f^0 &(x, \nu_\Omega)| u_n - h| + p_0^* ( u_n - h)  \bigg) \, d\mathcal{H}^{N-1}  \\
&+ \int_\Omega  f(x, \nabla u_n) \, dx + \int_\Omega  f^*(x, \overline{p}^*) \, dx - \langle \nabla u_n, -\overline{p}^* \rangle_{L^1(\Omega;\mathbb{R}^N),L^\infty(\Omega;\mathbb{R}^N)} \leq \varepsilon_n.
\end{align}
Because the trace of $u_n$ is fixed (and equal to the trace of $u$), the integral on $\partial\Omega$ does not change with $n$; hence, it has to equal zero. Thus, keeping in mind that $p_0^* = [-\overline{p}^*, \nu_\Omega]$ and $f^0(x,\nu_\Omega) \geq 0$, we get
\begin{equation}\label{BBCond}
[-\overline{p}^*, \nu_\Omega] \in \mbox{sign}(h-u) f^0(x,\nu_\Omega) \qquad  \mathcal{H}^{N-1}-\mbox{a.e. on} \ \partial \Omega.
\end{equation}
Hence, we have
\begin{equation}\label{inerior1}
0 \leq \int_\Omega  f(x,\nabla u_n) \, dx +  \int_\Omega  f^*(x, \overline{p}^*) \, dx \leq \int_\Omega \nabla u_n \cdot (- \overline{p}^*) \, dx + \varepsilon_n.
\end{equation}
Finally, keeping in mind that $-\mathrm{div}(\overline{p}^*) = u-g$ and again using the fact that the trace of $u_n$ is fixed and equal to the trace of $u$, by the Gauss-Green formula we get
\begin{align}
\int_\Omega \nabla u_n \cdot \overline{p}^* \, dx &= - \int_\Omega u_n \, \mathrm{div}(\overline{p}^*) \, dx + \int_{\partial\Omega} u_n \, [\overline{p}^*, \nu_\Omega] \, d\mathcal{H}^{N-1} \\
&=\int_\Omega u_n \, (u -g) \, dx +\int_{\partial\Omega} u \, [\overline{p}^*, \nu_\Omega] \, d\mathcal{H}^{N-1}.
\end{align}
Then, applying again the Gauss-Green formula, we have
\begin{align}
\lim_{n \to \infty}  \int_\Omega \nabla u_n \cdot \overline{p}^* \, dx &= \int_\Omega u \, (u-g) \, dx +\int_{\partial\Omega} u \, [\overline{p}^*, \nu_\Omega] \, d\mathcal{H}^{N-1} \\
&= -  \int_\Omega u \,  \mathrm{div}(\overline{p}^*) \, dx +\int_{\partial\Omega} u \, [\overline{p}^*, \nu_\Omega] \, d\mathcal{H}^{N-1} = \int_\Omega (\overline{p}^*, Du).
\end{align}
But then, passing to the limit in equation \eqref{inerior1}, by Reschetnyak theorem (Theorem \ref{thm:reshetnyak}) we get that
\begin{equation}
\int_\Omega f(\cdot, \nabla u) \, dx + \int_{\overline{\Omega}} f^0\left(\cdot, \frac{d D^s u}{d|D^s u|}\right) d|D^s u| + \int_\Omega f^*(x,\overline{p}^*) \, dx = \int_\Omega -\overline{p}^* \cdot \nabla u \, dx + \int_\Omega (-\overline{p}^*, Du)^s.
\end{equation}
From this, the required results on the absolutely continuous part and the singular part follow. Indeed, by the definition of the dual function
\begin{equation}
\int_\Omega f(\cdot, \nabla u) \, dx + \int_\Omega f^*(x,\overline{p}^*) \, dx \geq \int_\Omega -\overline{p}^* \cdot \nabla u \, dx,
\end{equation}
and by Proposition \ref{prop:resultfrombs} we get that
\begin{equation}
\int_{\overline{\Omega}} f^0 \left(\cdot, \frac{d D^s u}{d|D^s u|}\right) d|D^s u| \geq -\int_\Omega \overline{p}^*\cdot Du^s,
\end{equation}
so both of these inequalities have to be equalities.  But, the first one means exactly that $-\overline{p}^* \in \partial_\xi f(x,\nabla u)$, which implies \eqref{e1DefD} for the choice $\mathbf{z} = - \overline{p}^*$, and the second one means that $- \overline{p}^*\cdot Du^s = f^0(x,D^s u)$, so \eqref{e3DefD} holds. We already have the divergence constraint \eqref{e2DefD} by \eqref{div1} and the boundary constraint \eqref{e4DefD} follows from \eqref{BBCond}.

Finally, by \cite[Proposition 2.11]{Brezis} we have
$$ D(\partial \mathcal{F}_h) \subset  D(\mathcal{F}_h) =  BV(\Omega) \cap L^2(\Omega) \subset \overline{D(\mathcal{F}_h)}^{L^2(\Omega)} \subset \overline{D(\partial \mathcal{F}_h)}^{L^2(\Omega)},$$
from which follows the density of the domain.
\end{proof}

Our concept of solutions of the Dirichlet problem \eqref{DirichletProb} is the following:

\begin{definition}\label{def:dirichlet1p}
{\rm  Given $u_0 \in L^2(\Omega)$, we say that $u$ is a {\it weak solution} of the Dirichlet problem \eqref{DirichletProb} in $[0,T]$, if $u \in C([0,T]; L^2(\Omega)) \cap W_{loc}^{1,2}(0, T; L^2(\Omega))$,   $u(0, \cdot) =u_0$, and for almost all $t \in (0,T)$
\begin{equation}\label{def:dirichletflow}
u_t(t, \cdot) + \mathcal{A}_h u(t, \cdot) \ni 0.
\end{equation}
In other words, if $u(t) \in BV(\Omega)$ and there exist vector fields $\z(t) \in X_2(\Omega)$ such that the following conditions hold:
\begin{equation}\label{e1DefPD}\z(t) \in \partial_\xi f(x,\nabla u(t)) \quad \mathcal{L}^N-\hbox{a.e. in } \Omega;\end{equation}
\begin{equation}\label{e2DefPD}  u_t(t) = \mathrm{div}(\z(t))\quad  \hbox{in} \ \mathcal{D}^\prime(\Omega);\end{equation}
\begin{equation}\label{e3DefPD} \z(t) \cdot D^s u(t) = f^0(x,D^s u(t)) = f^0(x,\overrightarrow{D^s u(t)}) \vert
D^s u(t) \vert \quad \hbox{as measures};\end{equation}
\begin{equation}\label{e4DefPD}[\z(t), \nu_\Omega] \in \mbox{sign}(h-u(t)) f^0(x,\nu_\Omega) \qquad  \mathcal{H}^{N-1}-\mbox{a.e. on } \partial \Omega.\end{equation}

}
\end{definition}

Then, using the classical theory of maximal monotone operators (see for instance \cite{Brezis}), as a consequence of Theorem \ref{thm:Dirichlet} we have the following existence and uniqueness theorem.

\begin{theorem}\label{ExistUnique}
Suppose that $f^0(x,\nu_\Omega) \geq 0$ $\mathcal{H}^{N-1}$-a.e. on $\partial\Omega$. Then, under the assumptions \ref{A1}-\ref{A2}, for any $u_0 \in L^2(\Omega)$ and all $T > 0$ there exists a unique weak solution of the Dirichlet problem \eqref{DirichletProb} in $[0,T]$.
\end{theorem}

Definition \ref{def:dirichlet1p} agrees with the classical definitions for the total variation flow and functionals of class $C^1$ similarly as in Examples \ref{ex:tvflow} and \ref{ex:c1functional}, with the only difference in the last condition (on the boundary behaviour). Below, we present another example, which concerns the gradient flow of the area functional.

\begin{example}{\rm
In the particular case of the nonparametric area integrand
$$f(x, \xi) = \sqrt{1 + \Vert \xi \Vert^2},$$
we have $f^0(x,\xi) = \vert \xi \vert$ and $\a(x, \xi) = \partial_\xi f(x, \xi) = \frac{\xi}{\sqrt{1+ \vert \xi \vert^2}}$. Therefore, the definition of a weak solution to the time-dependent minimal surface equation
\begin{equation}\label{timedepend}
\left\{ \begin{array}{lll} u_t (t,x) = {\rm div}\,  \left( \frac{Du(t,x)}{\sqrt{1+ \vert Du(t,x)\vert^2}} \right)   \quad &\hbox{in} \ \ (0, T) \times \Omega, \\[5pt]u(t) = h \quad &\hbox{on} \ \ (0, T) \times  \partial\Omega, \\[5pt] u(0,x) = u_0(x) \quad & \hbox{in} \ \  \Omega, \end{array} \right.
\end{equation}
takes the following form. Given $u_0 \in L^2(\Omega)$, we say that $u$ is a {\it weak solution} of the Dirichlet problem \eqref{timedepend} in $[0,T]$, if $u \in  C([0,T]; L^2(\Omega)) \cap W_{loc}^{1,2}(0, T; L^2(\Omega))$,   $u(0, \cdot) =u_0$, and for almost all $t \in (0,T)$ we have $u(t) \in BV(\Omega)$ and there exist vector fields $\z(t) \in X_2(\Omega)$ such that the following conditions hold:
\begin{equation}\label{e1Timed}\z(t) \in \frac{\nabla u(t,x)}{\sqrt{1+ \vert \nabla u(t,x)\vert^2}} \quad \mathcal{L}^N-\hbox{a.e. in } \Omega;\end{equation}
\begin{equation}\label{e2Timed}
u_t(t) = \mathrm{div}(\z(t))\quad  \hbox{in} \ \mathcal{D}^\prime(\Omega);
\end{equation}
\begin{equation}\label{e3Timed}
\z(t) \cdot D^s u(t) = \vert D^s u(t) \vert \quad \hbox{as measures};
\end{equation}
\begin{equation}\label{e4Timed}
[\z(t), \nu_\Omega] \in \mbox{sign}(h-u(t)) \qquad  \mathcal{H}^{N-1}-\mbox{a.e. on } \partial \Omega.
\end{equation}
This concept of solution coincides with the one given by Demengel and Temam in \cite{DemengelTeman1} for the nonparametric area integrand. Moreover, our existence result (Theorem \ref{ExistUnique}) agrees with \cite[Theorem 3.1]{DemengelTeman1}.  Let us note that the characterisation of solutions in \cite{DemengelTeman1} was also obtained using convex duality, but with a different technique from the one presented in this paper.}
\end{example}

\bigskip

\noindent {\bf Acknowledgments.} The first author has been partially supported by the Austrian Science Fund (FWF), grant 10.55776/ESP88, and the OeAD-WTZ project CZ 01/2021. The second author has been partially supported by the Conselleria d'Innovaci\'{o}, Universitats, Ci\`{e}ncia y Societat Digital, project AICO/2021/223. For the purpose of open access, the
authors have applied a CC BY public copyright licence to any Author Accepted Manuscript version arising from this submission.


\begin{thebibliography}{00}


\bibitem{AFP}
L. Ambrosio, N. Fusco and D. Pallara,
\newblock {\it Functions of Bounded Variation
and Free Discontinuity Problems},
\newblock Oxford Mathematical Monographs, 2000.


\bibitem{ABCM1}
F. Andreu, C. Ballester, V. Caselles and J.M. Maz\'on,
\newblock {\it Minimizing total variation flow},
\newblock  Diff. and Int. Eq. {\bf 14} (2001), 321-360.

\bibitem{ABCM2}
F. Andreu, C. Ballester, V. Caselles and J.M. Maz\'on,
\newblock {\it The Dirichlet problem for the total variational flow},
\newblock  J. Funct. Anal. {\bf 180} (2001), 347-403.

\bibitem{ACMIbero}
F.~Andreu, V.~Caselles, and J.M. Maz\'on,
\newblock {\it A parabolic quasilinear
  problem for linear growth functionals},
\newblock   Rev. Mat. Iberoamericana {\bf 18} (2002), 135-185.


\bibitem{ACM4:01}
F.~Andreu, V.~Caselles, and J.M. Maz\'on,
\newblock {\it Existence and uniqueness of solution for a parabolic quasilinear
  problem for linear growth functionals with {$L^1$} data},
\newblock  Math. Ann. {\bf 322} (2002), 139-206.

\bibitem{ACMJEE}
F.~Andreu, V.~Caselles, and J.M. Maz\'on,
\newblock {\it The Cauchy problem for linear growth functionals},
\newblock   J. Evol. Equat. {\bf 3} (2003), 39-65.

		
\bibitem{ACMBook}
F.~Andreu, V.~Caselles, and J.M. Maz\'on, \newblock {\it Parabolic Quasilinear	Equations Minimizing Linear Growth Functionals}, Progress in Mathematics, vol. 223, Birkh\"auser, 2004.



\bibitem{Anzellotti1}
G. Anzellotti,
\newblock {\it Pairings between measures and bounded functions
and compensated compactness},
\newblock Ann. di Matematica Pura ed Appl. IV (135)
(1983), 293-318.

\bibitem{Anzellotti2}
G. Anzellotti, {\it The Euler equation for functionals with linear growth},
 Trans. Amer. Math. Soc. {\bf 290}  (1985), 483-500.


\bibitem{Anzellotti3}
G. Anzellotti, {\it On the minima of functionals with linear growth}, Rend. Semin. Mat. Univ. Padova {\bf 75} (1986) 91-110.


\bibitem{Anzellotti4}
G. Anzellotti, {\it Traces of bounded vector fields and the divergence theorem}, preprint, 1983, 43 pages.

		
\bibitem{BS} L. Beck, T. Schmidt, {\it Convex duality and uniqueness for BV-minimizers}, J. Funct. Anal. {\bf 268} (2015), 3061-3107.


\bibitem{CFM} V. Caselles, G. Facciolo  and E. Meinhardt, {\it Anisotropic Cheeger sets and applications}, SIAM J. Imaging Sci. {\bf 2} (2009), 1211--1254.


\bibitem{DHKW} U. Dierkes, S. Hildebrandt, A. K\"uster and O. Wohrab, {\it Minimal Surfaces}, Vol I, II, Springer-Verlag, 1992.

\bibitem{Brezis} H. Brezis, {Operateurs Maximaux Monotones}, North Holland, Amsterdam, 1973.

\bibitem{CCCM} G.E. Comi, G. Crasta, V. De Cicco and A. Malusa, {\it Representation formulas for pairings between divergence-measure fields and BV functions}, J. Funct. Anal. {\bf 286} (2024), 110192.

\bibitem{DemengelTeman1} F. Demengel and R. Temam, {\it Convex functions of a measure and applications}, Indiana Univ. Math. J. {\bf 33} (1984), 673-709.

\bibitem{EkelandTemam} I. Ekeland, R. Temam, {\it Convex analysis and variational problems}, North-Holland Publ. Company, Amsterdam, 1976.

\bibitem{EG} L. C. Evans and R. F. Gariepy, \newblock {\it Measure Theory and Fine Properties of Functions}, \newblock Studies in Advanced Math., CRC Press, 1992.

\bibitem{Giusti} E. Giusti, {\it Minimal Surfaces and Functions of Bounded Variation}, Birkh\"auser, Basel, 1984.

\bibitem{GM2021-3} W. G\'{o}rny and J.M. Maz\'on, {\it The Neumann and Dirichlet problems for the total variation flow in metric measure spaces}, Adv. Calc. Var. {\bf 17} (2024), 131--164.

\bibitem{HardtZhou} R. Hardt and X. Zhou, \newblock {\it An evolution problem for linear growth functionals}, \newblock Commun. Partial Differential Equations {\bf 19} (1994), 1879-1907.

\bibitem{KohnTeman} R. Kohn and R. Temam,
\newblock {\it Dual space of stress and strains with application to Hencky plasticity},
\newblock Appl. Math. Optim. {\bf 10} (1983), 1-35.

\bibitem{LT} A. Lichnewsky and R. Temam, {\it Pseudosolutions of the time-dependent minimal surface inequality}, J. Diff. Equ. {\bf 30} (1979), 340--364.


\bibitem{Moll} S. Moll, {\it The anisotropic total variation flow}, Math. Ann. {\bf 332} (2005), 177--218.

\bibitem{RS} M. Reed and B. Simon, {Methods of Modern Mathematical Physics,} vol. 1, Functional Analysis, Academic Press, San Diego, 1980.

\bibitem{Zhou} X. Zhou,
\newblock {\it An evolution problem for plastic antiplanar shear},
\newblock Appl. Math. Optim. {\bf 25} (1992), 263-285.

\bibitem{Ziemer} W. P. Ziemer,
\newblock {\it Weakly Differentiable Functions},
\newblock GTM 120, Springer Verlag, 1989.


\end{thebibliography}
\end{document}